\documentclass[11pt]{article}

\usepackage{geometry,cite,enumerate}

\usepackage{amssymb,amsmath,amsthm, amsfonts}

\newtheorem{theorem}{Theorem}[section]
\newtheorem{proposition}[theorem]{Proposition}

\newtheorem{lemma}[theorem]{Lemma}
\newtheorem{corollary}[theorem]{Corollary}
\newtheorem{remark}[theorem]{Remark}

\newcommand{\eps}{\varepsilon}

\newcommand{\R}{\mathbb R}

\numberwithin{equation}{section}

\title{A universal thin film model for Ginzburg-Landau energy with
  dipolar interaction}

\author{Cyrill B. Muratov\thanks {Department of Mathematical Sciences,
    New Jersey Institute of Technology, Newark, NJ 07102, USA}}

\date{}

\begin{document}

\maketitle

\begin{center}
  {\em Dedicated to V. V. Osipov on the occasion of his 77th birthday}
\end{center}

\begin{abstract}
  We present an analytical treatment of a three-dimensional
  variational model of a system that exhibits a second-order phase
  transition in the presence of dipolar interactions. Within the
  framework of Ginzburg-Landau theory, we concentrate on the case in
  which the domain occupied by the sample has the shape of a flat thin
  film and obtain a reduced two-dimensional, non-local variational
  model that describes the energetics of the system in terms of the
  order parameter averages across the film thickness. Namely, we show
  that the reduced two-dimensional model is in a certain sense
  asymptotically equivalent to the original three-dimensional model
  for small film thicknesses. Using this asymptotic equivalence, we
  analyze two different thin film limits for the full
  three-dimensional model via the methods of $\Gamma$-convergence
  applied to the reduced two-dimensional model. In the first regime,
  in which the film thickness vanishes while all other parameters
  remain fixed, we recover the local two-dimensional Ginzburg-Landau
  model. On the other hand, when the film thickness vanishes while the
  sample's lateral dimensions diverge at the right rate, we show that
  the system exhibits a transition from homogeneous to spatially
  modulated global energy minimizers. We identify a sharp threshold
  for this transition.
\end{abstract}


\section{Introduction}
\label{sec:introduction}

This paper is concerned with the behavior of ground states in systems
exhibiting a second-order phase transition which gives rise to the
emergence of dipolar order. A prototypical example may be found in
strongly uniaxial ferromagnets, such as magnetic garnet films with
perpendicular easy axis \cite{hubert,chikazumi,malozemoff,
  landau8}. In such films, spontaneous magnetization appears below the
Curie temperature due to ferromagnetic exchange, with the magnetic
moments of the electrons aligning in the direction normal to the film
plane. However, this local ordering is frustrated by the weak
dipole-dipole coupling, which instead favors anti-parallel alignment
of distant magnetic moments. Under appropriate conditions, this
competition between the short-range attractive and long-range
repulsive interactions is well known to produce various types of
inhomogeneous spatial patterns of magnetization, often referred to as
``modulated phases'' \cite{hubert,seul95,andelman09}. Other
physical systems with similar behavior include uniaxial ferroelectrics
\cite{landau8,strukov}, ferrofluids \cite{rosensweig} and Langmuir
layers \cite{seul95,andelman87}.

Within the mean-field approximation, these types of systems are
usually modeled by an appropriate free energy functional that contains
non-local terms coming from the dipolar interaction. Spatially
modulated phases are interpreted as either local or global minimizers
of the respective energy functional. A phase diagram is then
established by comparing the energies of the candidate ``phases'' and
selecting those corresponding to the global minimum of the
energy. Mathematically, this leads to a formidable variational
problem, which has been well known to exhibit intricate dependence on
the model parameters and geometry because of its non-convex and
non-local character. In the context of micromagnetics, a whole zoo of
different behaviors have been recently established (see, e.g.,
\cite{choksi98,choksi99,kohn05arma,desimone06r,otto10,km:jns11}; this
list is certainly very far from exclusive).

The complexity of the problem may be somewhat reduced near a phase
transition point, where the energy functional attains an
asymptotically universal form coming from the Landau expansion (still
within the mean-field approximation). This is the approach taken by
\cite{garel82,roland90,seul95,m:phd,m:pre02,jagla04,andelman09}, which
is also adopted by us here.  We start by formulating the
three-dimensional Ginzburg-Landau theory of a system undergoing a
second-order phase transition, in which the order parameter is
associated with dipolar ordering (for a recent review of the general
Ginzburg-Landau formalism, see \cite{hohenberg15}; for a stochastic
perspective, see also \cite{demasi94,mourrat16}). We then derive a
reduced two-dimensional Ginzburg-Landau theory with a modified
non-local term which becomes asymptotically equivalent to the full
three-dimensional theory as the film thickness vanishes. This
reduction is done in the spirit of $\Gamma$-development
\cite{braides08} and is the main result of the paper.

Consider a region $\Omega \subset \R^3$ occupied by the material and
assume that this region is in the shape of a film of thickness
$\delta > 0$, cross-section $D \subset \R^2$ and rounded edges (a
pancake-shaped domain). Namely, we assume\footnote{Recall that
  $D + B_\delta = \mathbf \{ \mathbf r \in \R^2 : \text{dist}(\mathbf
  r, D) < \delta\}$.}
that
$D \times (0, \delta) \subset \Omega \subset (D + B_\delta) \times (0,
\delta)$
and both $D$ and $\Omega$ have boundaries of class $C^2$. Note that we
do not necessarily assume that $D$ is connected.  We are particularly
interested in the case when $\delta$ is sufficiently small,
corresponding to a {\em thin film} (how small the value of $\delta$
should be in order for a film to be considered as thin will be
discussed later).  Inside $\Omega$, the state of the material is
described by a scalar order parameter $\phi = \phi(\mathbf r)$, where
$\mathbf r = (x, y, z) \in \R^3$ stands for the spatial
coordinate. The order parameter represents the magnitude of the
magnetization or polarization vector in the $z$-direction. In the
following, we extend $\phi$ by zero outside $\Omega$. Then the
Ginzburg-Landau free energy plus the dipolar interaction energy can be
written in the following form \cite{landau8}:
\begin{align}
  \label{Fdim}
  {\mathcal F(\phi) \over k_B T_c} = \int_\Omega \left( {g \over 2}
  |\nabla \phi|^2 + {a \over 2} (T - T_c) \phi^2 + {b \over 4} \phi^4
  - h \phi \right) d^3 r + {c \over 2} \int_{\R^3} \partial_z \phi 
  (-\Delta)^{-1} \partial_z \phi \, d^3 r. 
\end{align}
Here, $k_B$ is the Boltzmann constant, $T$ is temperature, $T_c$ is
the transition temperature in the absence of the dipolar interaction,
$h = h(x, y)$ is the applied field normal to the film plane, and
$a, b, c, g$ are positive material constants. Also, the symbol
$(-\Delta)^{-1}$ stands for the convolution with the Newtonian
potential $1 / ( 4 \pi |\mathbf r|)$ in three space dimensions, and
the derivative $\partial_z \phi$ in $\R^3$ is understood
distributionally.

When $\delta$ is small, the gradient term is expected to strongly
penalize the variations of $\phi$ in the $z$-direction. Furthermore,
it is easy to see that to the leading order the dipolar term should
become local. Indeed, since for small $\delta$ we have
$\Delta \approx \partial^2 z$ in a certain sense, the energy in
\eqref{Fdim} may be equivalently rewritten as
\begin{align}
  \label{Fdim2}
  {\mathcal F(\phi) \over k_B T_c} = \int_\Omega \left( {g \over 2}
  |\nabla \phi|^2 + {a \over 2} (T - T_c^*) \phi^2 + {b \over 4} \phi^4
  - h \phi \right) d^3 r 
  + {c \over 2} \int_{\R^3} \left( \partial_z \phi 
  (-\Delta)^{-1} \partial_z \phi - \phi^2 \right) \, d^3 r,
\end{align}
where we introduced the renormalized critical temperature
$T_c^* = T_c - \frac{c}{a}$ that contains the contribution of the
dipolar interaction and rewrote the last term so that it is expected
to be $o(\delta)$ as $\delta \to 0$.  Note that in the context of
micromagnetics, such an argument was made rigorous by Gioia and James
\cite{gioia97} (see also the following sections). Furthermore,
plugging in a $z$-independent ansatz $\phi(x, y, z) = \bar\phi(x, y)$,
where $\bar\phi : D \to \R$ is sufficiently smooth (extended by zero
outside $D$), one straightforwardly obtains (here and everywhere below
we use $\mathbf r$ to denote either a point in $\R^3$ or $\R^2$,
depending on the context)
\begin{align}
  \label{Fdimbar}
  {\mathcal F(\phi) \over k_B T_c} 
  & = \delta \int_D \left( {g \over 2}
    |\nabla \bar \phi|^2 + {a \over 2} (T - T_c^*) \bar \phi^2 + {b
    \over 4} \bar \phi^4
    - h \bar \phi \right) d^2 r  + O(\delta^2) \notag \\
  & + {c \over 4 \pi} \int_{\R^2} \int_{\R^2} \left( {1 \over |\mathbf r
    - \mathbf r'|} - {1 \over \sqrt{|\mathbf r
    - \mathbf r'|^2 + \delta^2 } } -  2 \pi \delta^{(2)}(\mathbf r
    - \mathbf r') \, \delta \right) \bar \phi(\mathbf r) \bar \phi
    (\mathbf r') \, d^2 r \, d^2 r',
\end{align}
where $\delta^{(2)}(\mathbf r)$ is the two-dimensional Dirac
delta-function. Formally expanding the integrand in the last term in
\eqref{Fdimbar} in the powers of $\delta$, one can then see that to
the leading order the kernel becomes
$\delta^2 / (8 \pi |\mathbf r - \mathbf r'|^3)$. In the physics
literature, this approximation is often adopted to arrive at a leading
order asymptotic theory for thin films with dipolar interactions, with
the $1/r^3$ kernel representing the dipole-dipole repulsion (as is
done, e.g., in the review \cite{andelman09}). This, however, is
incorrect, since the $1/r^3$ kernel is too singular in two dimensions,
and thus the resulting double integral does not make sense. A more
sound approach mathematically is to go to Fourier space, perform an
expansion there and then invert the transform. This leads to the
following formula:
\begin{align}
  \label{Fdimbar0}
  {\mathcal F(\phi) \over k_B T_c \delta} 
  \approx \int_D \left( {g \over 2}
  |\nabla \bar \phi|^2 + {a \over 2} (T - T_c^*) \bar \phi^2 + {b
  \over 4} \bar \phi^4
  - h \bar \phi \right) d^2 r  \notag \\
  - {c \delta \over 16 \pi} \int_{\R^2} \int_{\R^2}  {(\bar\phi(\mathbf
  r) - \bar\phi(\mathbf r'))^2 \over |\mathbf r - \mathbf r'|^3} \,
  d^2 r \, d^2 r'.
\end{align}
Contrary to the previous case, the last integral in the right-hand
side of \eqref{Fdimbar0} is well defined, at least for smooth
functions vanishing on $\partial D$. Moreover, since this term can be
interpreted, up to a constant factor, as the homogeneous
$H^{1/2}$-norm squared of $\bar\phi$ (see, e.g., \cite{dinezza12}),
one can write \eqref{Fdimbar0} as \cite{andelman87,jagla04}
\begin{align}
  \label{Fdimbar00}
  {\mathcal F(\phi) \over k_B T_c \delta} 
  \approx \int_D \left( {g \over 2}
  |\nabla \bar \phi|^2 + {a \over 2} (T - T_c^*) \bar \phi^2 + {b
  \over 4} \bar \phi^4
  - h \bar \phi \right) d^2 r  
  - {c \delta \over 4} \int_{\R^2} \bar\phi (-\Delta)^{1/2} \bar\phi
  \, d^2 r,
\end{align}
where the half-Laplacian operator $ (-\Delta)^{1/2}$ is understood as
a map whose Fourier symbol is $|k|$, or, equivalently, as an integral
operator whose action on smooth functions with compact support is
given by \cite{dinezza12}
\begin{align}
  \label{halflapl}
  (-\Delta)^{1/2} \bar\phi(\mathbf r) = {1 \over 4 \pi} \int_{\R^2}
  {2 \bar\phi(\mathbf r) - \bar\phi(\mathbf r - \mathbf z) -
  \bar\phi(\mathbf r + \mathbf z) \over |\mathbf z|^3} \, d^2 z \qquad
  \mathbf r \in \R^2.  
\end{align}
In particular, since $D$ is assumed to be bounded, we must necessarily
have $\bar\phi \in H^1(D)$ in order for the right-hand side of
\eqref{Fdimbar0} to be less than $+\infty$. If also
$\bar\phi \in H^1_0(D)$, then by interpolation the energy is bounded
below and is thus well defined \cite{lieb-loss}. Yet, there is still
an issue with the expression for the energy in \eqref{Fdimbar0}, which
becomes negative infinity as soon as $\bar \phi$ does not vanish at
the boundary of $D$. This issue is quite severe and exists even for
$\bar \phi = const$ in $D$. The reason for the latter is that the
energy in \eqref{Fdimbar0} fails to capture a reduced local
contribution of the dipoles near the boundary, since only half of the
neighbors are present at $\partial D$. In the following, we fix this
issue by introducing a smooth cutoff near the boundary of $D$ in
computing the last term in the right-hand side of
\eqref{Fdimbar0}. This allows us to estimate, under appropriate
assumptions, the original energy from \eqref{Fdim} from below by a
reduced energy similar to the one in \eqref{Fdimbar0} evaluated on the
average of the order parameter in the $z$-direction, with the relative
error controlled only by $\delta$ (for precise statements, see the
following section). Since the latter energy is also a good
approximation to the original energy for $z$-independent
configurations, this then allows us to make a number of conclusions
regarding the energy minimizers of the full energy in \eqref{Fdim}
defined on three-dimensional configurations. Thus, understanding the
behavior of the energy minimizers for \eqref{Fdim} can be achieved by
looking at a somewhat simpler energy of the type in \eqref{Fdimbar0},
which, nevertheless, retains most of the complexity of the former.

To summarize, in this paper we show that the energy in
\eqref{Fdimbar0} is in a certain sense asymptotically equivalent to
the energy in \eqref{Fdim} {\em without} assuming that the order
parameter does not vary in the $z$-direction. Instead, we show that
the energy in \eqref{Fdimbar0} correctly describes the energetics of
the low-energy three-dimensional order parameter configurations in
terms of their $z$-averages. More precisely, under some technical
assumptions the energy in \eqref{Fdimbar0} evaluated on the
$z$-average of the order parameter gives an asymptotically accurate
lower bound for the full energy in \eqref{Fdim} evaluated on the
three-dimensional order parameter configuration. On the other hand,
extending a two-dimensional order parameter configuration to a
three-dimensional $z$-independent configuration, one gets a value of
the full energy in \eqref{Fdim} that is asymptotically bounded above
by the value of the reduced energy in \eqref{Fdimbar0} evaluated on
the two-dimensional configuration. We note that the first result in
that direction was obtained by Kohn and Slastikov in the context of
micromagnetics, see \cite[Lemma 3]{kohn05arma}. Our analysis differs
from that in \cite{kohn05arma} in that it identifies the first two
non-trivial leading order terms in the expansion of the dipolar energy
in $\delta$ and provides sharp universal estimates for the remainder.

The main result of this paper on the asymptotic equivalence of the two
energies is presented in Theorem \ref{t:main}. This theorem relies on
key Lemma \ref{l:ded}, which establishes matching upper and lower
bounds for the dipolar energy of three-dimensional order parameter
configurations in terms of a non-local energy functional evaluated on
the $z$-averages in the plane, with the error controlled by the
Dirichlet energy with a vanishingly small coefficient as the film
thickness becomes small. This produces errors that can be controlled
by the $L^\infty$ norm of the order parameter, apart from some
possible additional contributions near the film edge in the upper
bounds. Notice that boundedness of the $L^\infty$ norm of both the
three- and two-dimensional energy minimizing order parameter
configurations is a reasonable assumption in view of the regularity of
minimizers established in Propositions \ref{p:exist3d} and
\ref{p:exist2d}. We also note that a uniform $L^\infty$ bound by the
equilibrium value of the order parameter is a fairly standard
assumption for the ansatz-based computations in the physics literature
and is a property which is also observed in some numerical simulations
(see, e.g., \cite{garel82,roland90,kaplan93,ng95,jagla04}).

With the reduced energy identified, we proceed to analyze two thin
film regimes. In the first regime, only the film thickness is sent to
zero, with all the other parameters as well as the film cross-section
fixed. In the context of micromagnetics, such a result was first
obtained by Gioia and James in \cite{gioia97}. Here under a uniform
$L^\infty$ bound this type of result follows immediately from Theorem
\ref{t:main}. Still, we are able to relax the $L^\infty$ constraint
and prove the result in the full generality by establishing
$\Gamma$-convergence of the full energy to the {\em local} energy
evaluated on the $z$-averages, see Theorem \ref{t:gioia}. Here the
proof requires a different treatment of the non-local contributions to
the energy near the film edge.

Finally, we consider a regime in which simultaneously the film
thickness goes to zero, while the film's lateral dimension goes to
infinity with a suitable rate that is exponential in the film
thickness. We note that these types of scalings were previously
discussed in the physics literature \cite{kaplan93,ng95} and have been
recently treated by Kn\"upfer, Muratov and Nolte within the framework
of micromagnetics \cite{kmn17}. In this regime, after a rescaling that
fixes the domain in the plane we prove a $\Gamma$-convergence result
for the reduced energy in Theorem \ref{t:nolte}. Together with Theorem
\ref{t:main}, this result then gives asymptotic non-existence of
non-trivial minimizers of the full energy, under a uniform $L^\infty$
bound and a technical assumption that the sample is maintained in a
single phase near the edge. We further identify a critical value of
the rescaled film thickness above which pattern formation occurs, see
Corollaries \ref{c:nolteno} and \ref{c:noltesi}. The proof relies on
the standard Modica-Mortola trick \cite{modica87} and an interpolation
Lemma \ref{l:interp} similar to the one obtained in the context of
thin film micromagnetics \cite{desimone06}, and follows closely the
arguments that lead to Theorem 3.5 in our companion paper
\cite{kmn17}. Note that combining Theorem \ref{t:nolte} with Theorem
\ref{t:main} yields an analog of Theorem 3.1 in \cite{kmn17}. A novel
aspect of Theorem \ref{t:nolte} is the consideration of the energy
contribution from the non-local term near the sample edge.

Our paper is organized as follows. In Sec. \ref{sec:main}, we present
the main results of the paper. In Sec. \ref{sec:preliminaries}, a
number of auxiliary results is obtained that are used throughout the
proofs. Here we also derive the Euler-Lagrange equations associated
with minimizers of the full and the reduced energies, see Propositions
\ref{p:exist3d} and \ref{p:exist2d}. Then, in Sec. \ref{sec:proof1} we
give the proof of Theorem \ref{t:main} and in
Sec. \ref{sec:proof-theor-reft:g}, we give the proof of Theorem
\ref{t:gioia}. Finally, in Sec. \ref{sec:rest-proof} we present the
proof of Theorem \ref{t:nolte} and Corollary \ref{c:nolteno}.

\section{Main results}
\label{sec:main}

We now turn to our main results. We start by carrying out a suitable
non-dimensionalization for the energy in \eqref{Fdim}. To that end, we
use instead the representation in \eqref{Fdim2} and choose the units
of length, $\phi$ and the energy in such a way that
$k_B T_c = a (T_c^* - T) = b = g = 1$, treating the most interesting
case $T < T_c^*$. Also, to simplify the presentation we set $h = 0$
throughout the rest of the paper. The external field $h$ can be
trivially added back in all the results below. 

Denoting the dimensionless dipolar strength by $\gamma > 0$, we write
the rescaled version of the energy in \eqref{Fdim2}, up to an additive
constant, as
\begin{align}
  \label{E}
  \mathcal E(\phi) := \int_\Omega \left( \frac12 |\nabla \phi|^2 +
  \frac14 \left( 1 -  \phi^2 \right)^2 \right) d^3 r + {\gamma \over
  2} \int_{\R^3} \left( \partial_z \phi  (-\Delta)^{-1} \partial_z
  \phi - \phi^2 \right) d^3 r,
\end{align}
where $\phi \in H^1(\Omega)$, extended by zero to
$\R^3 \backslash \Omega$. The energy $\mathcal E$ in \eqref{E} thus
depends on only two dimensionless parameters, $\delta$ and $\gamma$,
as well as on the domain $D$, whose diameter may have a relationship
with these two parameters when considering various asymptotic
regimes. The unit of length above is chosen so that the characteristic
length scale of variation of $\phi$ in the absence of the dipolar
interaction is of order unity. Therefore, the thin film regime that we
are interested in should correspond to $\delta \lesssim 1$. Note that
in terms of the original, dimensional variables, we have
\begin{align}
  \label{gamma}
  \gamma = {c \over a (T_c^* - T)}.
\end{align}
In the context of ferromagnetism, the parameter $\gamma$ may be both
small and large, depending on how close the value of $T$ is to
$T_c^*$.  Indeed, since the stray field interaction is a relativistic
effect in comparison with the exchange interaction driving the phase
transition, it should be considerably weaker than the latter away from
the critical temperature \cite{chikazumi}. At the same time, as $T$
approaches $T_c^*$, the value of $\gamma$ diverges.

We next introduce a cutoff function
$\chi_\delta \in C^\infty_c (\R^2)$.  Namely, we define
$\eta : \R \to [0,1]$ such that $\eta \in C^\infty(\R)$, $\eta(t) = 0$
for all $t \leq 1$, $\eta(t) = 1$ for all $t \geq 2$ and
$0 \leq \eta'(t) \leq 2$ for all $t \in \R$. We then define
$\chi_\delta(\mathbf r) = \eta( \delta^{-1} \text{dist}(\mathbf r,
\R^2 \backslash D))$.  We also define
\begin{align}
  \label{Omd}
  D_\delta := \left\{ \mathbf r \in D : \text{dist}(\mathbf r, \partial
  D) > \delta \right\} \qquad \text{and} \qquad \Omega_\delta :=
  D_\delta \times (0, \delta),
\end{align}
and note that $\overline D_\delta = \text{supp} (\chi_\delta)$.
Finally, with a slight abuse of notation we will also treat
$\chi_\delta$ as a $z$-independent function of all three coordinates,
depending on the context.

We now define the following {\em reduced energy} for
$\bar\phi \in H^1(D)$ and $\alpha > 0$:
\begin{align}
  \label{EE}
  E(\bar\phi) := \int_D \left( \frac12 \left( 1 - \alpha
  \delta^2 \right) |\nabla
  \bar\phi|^2 + \frac14 \left( 1 - \bar\phi^2 \right)^2 \right) d^2 r
  \qquad \qquad \notag \\ 
  - {\gamma \delta \over 16 \pi} \int_{\R^2} \int_{\R^2}
  {(\chi_\delta(\mathbf 
  r) \bar\phi(\mathbf r) - \chi_\delta(\mathbf
  r') \bar\phi(\mathbf 
  r'))^2 \over |\mathbf r - \mathbf r'|^3} \, d^2 r \, d^2 r'. 
\end{align}
This definition makes sense, because we have
$\chi_\delta \bar\phi \in H^1(\R^2)$ and, hence, by interpolation the
last term in \eqref{EE} is well-defined \cite{lieb-loss}. What we will
show below is that if
\begin{align}
  \label{av}
  \bar\phi(x, y) = {1 \over \delta} \int_0^\delta \phi(x, y, z) \, dz
  \qquad (x, y) \in D,
\end{align}
then with a suitable explicit choice of $\alpha$ the value of
$E(\bar\phi) \delta$ may be used to bound from below the value of
$\mathcal E(\phi)$, up to a small error in $\delta$. Conversely, the
value of $E(\bar\phi) \delta$ provides a good approximation for the
value of $\mathcal E(\phi)$, with a small relative error, when $\phi$
is chosen to be independent of $z$. We make this statement precise in
the following theorem.

\begin{theorem}
  \label{t:main}
  There exist universal constants $\alpha_1 > 0$, $\alpha_2 > 0$ and
  $\beta > 0$ such that for every $\delta > 0$ sufficiently small
  there holds:
  \begin{enumerate}[(i)]
  \item If $\phi \in H^1(\Omega) \cap L^\infty(\Omega)$ and $\bar\phi$
    is defined by \eqref{av}, then
    \begin{align}
      \label{EvsEElower}
      \mathcal E(\phi) \geq E(\bar\phi) \delta  - \beta \gamma
      \delta^2 \| \phi \|_{L^\infty(\Omega)}^2 |\partial D|,
    \end{align}
    with $\alpha = \alpha_1 + \gamma \alpha_2$.
  \item For every $\bar\phi \in H^1(D) \cap L^\infty(D)$ there exists
    $\phi \in H^1(\Omega) \cap L^\infty(\Omega)$ such that
    $\| \phi \|_{L^\infty(\Omega)} \leq \| \bar\phi \|_{L^\infty(D)}$,
    $\phi(x, y, z) = \bar \phi(x, y)$ for all $(x, y) \in D$, and
    \begin{align}
      \label{EvsEEupper}
      \mathcal E(\phi) 
      & \leq (1 - 2 \alpha \delta^2)^{-1} E(\bar\phi)
        \delta \notag \\
      & + \beta  \delta^2 (1 + \gamma^2) \left( 1 + \| \bar \phi
        \|_{L^\infty(\Omega)}^4 \right) |\partial D| +  \beta \delta \|
        \nabla \bar \phi \|_{H^1(D \backslash D_\delta)}^2 .  
    \end{align}
  \end{enumerate}
\end{theorem}

Note that for $\gamma \lesssim 1$ and
$\| \phi \|_{L^\infty(\Omega)} \lesssim 1$ the additive error term
appearing in both the upper and the lower bound in Theorem
\ref{t:main} is of the order of the dipolar self-interaction energy of
$\phi$ at the sample edge $\Omega \backslash \Omega_\delta$.  Thus,
the asymptotic equivalence of $\mathcal E$ and $E$ established in
Theorem \ref{t:main} holds when $|\mathcal E(\phi)| \gg \delta^2$,
when the bulk contribution to the energy dominates that of the
edge. Note that in this case the non-local term in $E$ is expected to
capture the leading $O(\delta^2 |\log \delta|)$ contribution to
$\mathcal E$ from the film edge. Hence, the additive error term
appearing in Theorem \ref{t:main} should still be negligible even when
the edge effects are prominent. We point out that a smooth cutoff near
the sample edge was recently used to model boundary effects in
computational micromagnetic studies of ultrathin ferromagnetic films,
a closely related problem \cite{mov:jap15}.

We now show how Theorem \ref{t:main} may be used to establish some of
the asymptotic properties of the energy minimizing configurations for
the original energy $\mathcal E$ as $\delta \to 0$ by studying the
reduced energy $E$. We begin by establishing a result similar to that
of Gioia and James for a closely related vectorial model of
micromagnetics in the thin film limit \cite{gioia97}. Namely, we
consider the simplest thin film regime, in which $\delta \to 0$ with
both $\gamma$ and $D$ fixed. In this regime, we show that the
energetics of the low energy configurations in the original
three-dimensional model can be asymptotically described via the local
two-dimensional energy. The proof for uniformly bounded sequences
follows by combining the result in Theorem \ref{t:main} with the
$\delta \to 0$ limit behavior of $E$ established in Proposition
\ref{p:EdE0}. A slight modification of the proof of Theorem
\ref{t:main} in this regime allows to remove the assumption of
boundedness, so below we state the result in its full generality.

For fixed $D$, consider a family of bounded open sets
$\Omega^\delta \subset \R^3$ such that
$D \times (0, \delta) \subset \Omega^\delta \subset (D + B_\delta)
\times (0, \delta)$.
Given $\phi_\delta \in H^1(\Omega^\delta)$, we define
$\bar\phi_\delta$ to be its $z$-average on $D$, i.e.,
$\bar\phi_\delta \in H^1(D)$ is defined by \eqref{av} with $\phi$
replaced by $\phi_\delta$.  We next define $\mathcal E_\delta$ to be
the family of functionals given by \eqref{E} with
$\Omega = \Omega^\delta$. We also define $E_0$ to be given by
\eqref{EE} with $\delta$ formally set to zero, i.e., we define
\begin{align}
  \label{E0}
  E_0(\bar\phi) := 
  \begin{cases}
    \displaystyle\int_D \left( \frac12 |\nabla \bar\phi|^2 + \frac14
      \left( 1 - \bar\phi^2 \right)^2 \right) d^2 r & \bar\phi \in
    H^1(D),
    \\
    +\infty & \text{otherwise}.
  \end{cases}
\end{align}
Then the following $\Gamma$-convergence result holds true (for a
general introduction to $\Gamma$-convergence, see, e.g.,
\cite{braides}).

\begin{theorem}
  \label{t:gioia}
  As $\delta \to 0$, we have
  \begin{align}
    \delta^{-1} \mathcal E_\delta \stackrel{\Gamma}\to E_0,
  \end{align}
  with respect to the $L^2$ convergence of the $z$-averages, in the
  following sense:
  \begin{enumerate}[(i)]
  \item For any sequence of $\delta \to 0$ and
    $\phi_\delta \in H^1(\Omega^\delta)$ such that
    $\| \nabla \phi_\delta \|_{L^2(\Omega)}^2 \leq C \delta$ for some
    $C > 0$ independent of $\delta$,
    $\bar\phi_\delta \rightharpoonup \bar\phi$ in $H^1(D)$ and
    $\bar\phi_\delta \to \bar\phi$ in $L^2(D)$, we have
  \begin{align}
    \label{liminfEg}
    \liminf_{\delta \to 0} \delta^{-1} \mathcal E_\delta(\phi_\delta)
    \geq E_0(\bar\phi). 
  \end{align}
\item For any $\bar\phi \in H^1(D)$ and every sequence of
  $\delta \to 0$, there exists $\phi_\delta \in H^1(\Omega^\delta)$
  such that $\| \nabla \phi_\delta \|_{L^2(\Omega)}^2 \leq C \delta$
  for some $C > 0$ independent of $\delta$,
  $\bar\phi_\delta \to \bar\phi$ in $L^2(D)$ and
  \begin{align}
      \label{limsupEg}
    \limsup_{\delta \to 0} \delta^{-1} \mathcal E_\delta(\phi_\delta)
    \leq E_0(\bar\phi).
  \end{align}
  \end{enumerate}
\end{theorem}

The assumption on the gradient in Theorem \ref{t:gioia} is a natural
assumption consistent with the scaling of the minimum energy for
$\mathcal E_\delta$. In particular, the theorem above applies, upon
extraction of subsequences, to $\phi_\delta \in H^1(\Omega^\delta)$
satisfying
\begin{align}
  \label{limsupE}
  \limsup_{\delta \to 0} \delta^{-1} \mathcal E_\delta(\phi_\delta) <
  +\infty, 
\end{align}
in view of the compactness of their $z$-averages in $H^1(D)$, see
Proposition \ref{p:comp3d}. Therefore, by Corollary \ref{c:nloc3dpos}
we have the following immediate consequence of Theorem \ref{t:gioia}
concerning global minimizers of $\mathcal E_\delta$. Note that the
latter exist for each $\delta > 0$ by Proposition \ref{p:exist3d}.

\begin{corollary}
  \label{c:gioia}
  Let $\phi_\delta \in H^1(\Omega^\delta)$ by a minimizer of
  $\mathcal E_\delta$. Then for any sequence of $\delta \to 0$ we have
  $\bar\phi_\delta \to \bar \phi$ in $L^2(D)$, where $\bar\phi$ takes
  a constant value $\pm 1$ in every connected component of $D$.
\end{corollary}

Let us point out that the addition to $\mathcal E_\delta(\phi)$ of an
applied field term $-\int_{\Omega^\delta} h \phi \, d^3 r$ with
$h = h(x, y) \in L^2(\Omega^\delta)$ does not change the
$\Gamma$-convergence result in Theorem \ref{t:gioia}, provided that
the term $-\int_D h \bar\phi \, d^2r$ is added to the definition of
$E_0$ in \eqref{E0}. Thus, as expected, in the thin film limit with
$D$ and $\gamma$ fixed one recovers the local Ginzburg-Landau energy
functional. We note, however, that physically the effect of the
dipolar interaction is still present in the renormalization of the
transition temperature from $T_c$ to $T_c^*$.

We finally turn our attention to a regime of practical interest in
which modulated patterns spontaneously emerge. In view of the previous
result, this requires simultaneous vanishing of the film thickness and
blowup of the film's lateral dimensions. To this end, we introduce a
small parameter $\eps > 0$ and consider domain $D^\eps = \eps^{-1} D$,
with a fixed bounded open set $D \subset \R^2$ with $C^2$ boundary
describing the shape of the film in the plane and lateral length scale
$\eps^{-1} \gg 1$. Next, we rescale all lengths with $\eps^{-1}$ and
define the rescaled domain $\Omega^\eps \subset \R^3$ occupied by the
material. Thus, for a film of thickness $\delta = \delta_\eps$ we have
$D \times (0, \eps \delta_\eps) \subset \Omega^\eps \subset (D +
B_{\eps\delta_\eps}) \times (0, \eps \delta_\eps)$.

In the rescaled variables, the energy in \eqref{E} takes the following
form:
\begin{align}
  \label{Eeps}
  \mathcal E_\eps(\phi) := \int_{\Omega^\eps} \left( \frac12 |\nabla
  \phi|^2 + \frac{1}{4 \eps^2} \left( 1 - \phi^2 \right)^2 \right) d^3
  r + {\gamma \over 2 \eps^2} \int_{\R^3} \left( \partial_z \phi
  (-\Delta)^{-1} \partial_z \phi - \phi^2 \right) d^3 r,
\end{align}
where $\phi \in H^1(\Omega^\eps)$ and the energy has been rescaled
with an overall factor $\eps$. Similarly, rescaling the energy in
\eqref{EE} with $\eps$ as well, for $\bar\phi \in H^1(D)$ we define 
\begin{align}
  \label{EEeps}
  E_\eps(\bar\phi) := \int_D \left( \frac{\eps}{2} (1 - \alpha 
  \delta_\eps^2) |\nabla \bar\phi|^2 + \frac{1}{4 \eps} \left( 1 -
  \bar\phi^2 \right)^2 \right) d^2 r \qquad \qquad \notag \\
  - {\gamma \delta_\eps \over 16 \pi} \int_{\R^2} \int_{\R^2}
  {(\chi_{\eps\delta_\eps} (\mathbf r) \bar\phi(\mathbf r) -
  \chi_{\eps \delta_\eps} (\mathbf r') \bar\phi(\mathbf r'))^2 \over
  |\mathbf r - \mathbf r'|^3} \, d^2 r \, d^2 r'. 
\end{align}
Notice that the overall factor of $\eps$ in the energy scale for both
energies above is chosen to obtain the Modica-Mortola scaling
\cite{modica87} in the reduced two-dimensional energy $E_\eps$, in
anticipation of its limit behavior as $\eps \to 0$.

With these notations, the lower bound in Theorem \ref{t:main} for
$\phi_\eps \in H^1(\Omega^\eps) \cap L^\infty(\Omega^\eps)$ satisfying
$\| \phi_\eps \|_{L^\infty(\Omega^\eps)} \leq M$ for some $M \geq 1$
fixed and for all $\delta_\eps$ sufficiently small becomes
\begin{align}
  \label{Eepslb}
  \mathcal E_\eps(\phi_\eps) \geq E_\eps(\bar\phi_\eps) \delta_\eps -
  C \delta_\eps^2, 
\end{align}
where
\begin{align}
  \label{aveps}
  \bar\phi_\eps(x, y) = {1 \over \eps \delta_\eps} \int_0^{\eps
  \delta_\eps} \phi_\eps(x, y, z) \, dz \qquad (x, y) \in D,
\end{align}
and $C > 0$ depends only on $\gamma$, $D$ and $M$, for a suitable
choice of $\alpha$ depending only on $\gamma$.  Concentrating on the
bulk properties of the configurations, we further assume that the
order parameter is equal to its bulk equilibrium value near the film
edge and does not exceed it in magnitude throughout the film (a more
thorough analysis of the behavior of global minimizers as $\eps \to 0$
goes far beyond the scope of the present paper and will be treated
elsewhere). Hence, we set $M = 1$ and for $\rho > 0$ sufficiently
small fixed we assume that $\bar\phi_\eps = 1$ in
$D \backslash D_\rho$, where $D_\rho$ is as in \eqref{Omd}. In this
case the upper bound from Theorem \ref{t:main} reads for all
$\bar\phi_\eps \in H^1(D) \cap L^\infty(D)$ such that
$\| \bar\phi_\eps \|_{L^\infty(D)} = 1$ and $\bar\phi_\eps = 1$ in
$D \backslash D_\rho$, for all $\delta_\eps$ sufficiently small:
\begin{align}
  \label{Eepsub}
  \mathcal E_\eps(\phi_\eps) \leq (1 - 2 \alpha \delta_\eps^2)^{-1}
  E_\eps(\bar\phi_\eps) \delta_\eps + C \delta_\eps^2,
\end{align}
where $C > 0$ is as before and
$\phi_\eps \in H^1(\Omega^\eps) \cap L^\infty(\Omega^\eps)$ satisfies
$\phi_\eps(x, y, z) = \bar\phi_\eps(x, y)$ for all $(x, y) \in D$, and
$\| \phi_\eps \|_{L^\infty(\Omega^\eps)} = 1$. We note that related
ideas were used in \cite{m:cmp10} in the asymptotic analysis of the
two-dimensional Ohta-Kawasaki energy.

We now specify the scaling of $\delta_\eps$ with $\eps$ for which
modulated patterns emerge. This scaling has been recently identified
in \cite{kmn17} in the studies of a closely related model from
micromagnetism. For $\lambda > 0$ fixed, we set
\begin{align}
  \label{deps}
  \delta_\eps = {\lambda \over \gamma |\ln \eps|}, 
\end{align}
and consider the limit behavior of the energies in \eqref{Eeps} and
\eqref{EEeps} as $\eps \to 0$. In \cite{kmn17}, a critical value of
$\lambda = \lambda_c$ has been identified, below which no modulated
patterns emerge as energy minimizers in this limit, while above this
value pattern formation occurs. A similar phenomenon takes place in
our problem, too. In the subcritical regime, the conclusion above is a
consequence of the following $\Gamma$-convergence result. In our case,
the threshold value of $\lambda$ is
\begin{align}
  \label{lamc}
  \lambda_c := {2 \pi \sqrt{2} \over 3}. 
\end{align}
We also define the constants
\begin{align}
  \label{sig01}
  \sigma_0 = {2 \sqrt{2} \over 3}, \qquad \sigma_1 = {1 \over \pi},
\end{align}
and notice that $\lambda_c = \sigma_0 / \sigma_1$. The following
theorem is a close analog of Theorem 3.5 in \cite{kmn17} obtained in a
periodic setting.

\begin{theorem}
  \label{t:nolte}
  Let $\rho > 0$, $0 < \lambda < \lambda_c$ and let $E_\eps$ be
  defined by \eqref{EEeps} with $\delta_\eps$ given by
  \eqref{deps}. Then as $\eps \to 0$ we have
  \begin{align}
    E_\eps \stackrel{\Gamma}{\to} E_*, \qquad E_*(\bar\phi) :=
    -\frac14 \sigma_1 \lambda |\partial D| +\frac12  (\sigma_0 -
    \sigma_1 \lambda) \int_D |\nabla \bar\phi| \, d^2 r,
  \end{align}
  where $\bar\phi \in BV(D; \{-1,1\})$, with respect to the $L^1(D)$
  convergence, in the following sense:
  \begin{enumerate}[(i)]
  \item For every sequence of
    $\bar\phi_\eps \in H^1(D) \cap L^\infty(D)$ such that
    $\bar\phi_\eps = 1$ in $D \backslash D_\rho$,
    $\| \bar\phi_\eps \|_{L^\infty(D)} = 1$, and
    \begin{align}
      \limsup_{\eps \to 0} E_\eps(\bar\phi_\eps) <
      +\infty,
    \end{align}
    there exists a subsequence (not relabelled) such that
    $\bar\phi_\eps \to \bar\phi$ in $L^1(D)$ and
    \begin{align}
      \label{liminfEeps}
      \liminf_{\eps \to 0} E_\eps(\bar\phi_\eps) \geq E_*(\bar\phi),
    \end{align}
    for some $\bar\phi \in BV(D; \{-1,1\})$ such that $\bar\phi = 1$
    in $D \backslash D_\rho$.
  \item For any $\bar\phi \in BV(D; \{-1,1\})$ such that
    $\bar\phi = 1$ in $D \backslash D_\rho$ there exists a sequence of
    $\bar\phi_\eps \in H^1(D) \cap L^\infty(D)$ such that
    $\bar\phi_\eps = 1$ in $D \backslash D_\rho$,
    $\| \bar\phi_\eps \|_{L^\infty(D)} = 1$,
    $\bar\phi_\eps \to \bar\phi$ in $L^1(D)$ and
   \begin{align}
     \label{limsupEeps}
     \limsup_{\eps \to 0} E_\eps(\bar\phi_\eps) \leq E_*(\bar\phi).
    \end{align}
  \end{enumerate}
\end{theorem}

\begin{remark}
  \label{r:nolte}
  The inequalities in \eqref{liminfEeps} and \eqref{limsupEeps} remain
  true for $\lambda \geq \lambda_c$ as well, if one assumes that
  $\bar \phi_\eps \to \bar \phi$ in $BV(D)$ in addition to
  $\bar\phi_\eps = 1$ in $D \backslash D_\rho$ and
  $\| \bar\phi_\eps \|_{L^\infty(D)} = 1$. However, the compactness
  statement of Theorem \ref{t:nolte} no longer holds for
  $\lambda > \lambda_c$ (for more details in a periodic setting, see
  \cite{kmn17}).
\end{remark}

Theorem \ref{t:nolte} implies, in particular, that for
$\lambda < \lambda_c$ all minimizers of $E_\eps$ among functions
$\bar\phi_\eps \in H^1(D) \cap L^\infty(D)$ satisfying
$\| \bar \phi_\eps \|_{L^\infty(D)} = 1$ and $\bar \phi_\eps = 1$ in
$D \backslash D_\rho$ for some $\rho > 0$ converge a.e. to
$\bar\phi = 1$ in $D$ as $\eps \to 0$, implying that minimizers within
this class approach a monodomain state for all $\eps$ sufficiently
small. This is consistent with the result in Corollary \ref{c:gioia}
in the other scaling regime considered earlier. As was already noted,
relaxing the assumption of boundedness and the behavior near the edge
to make the same conclusion about the unconstrained minimizers of
$E_\eps$ would require a rather delicate analysis of the energy
minimizing configurations near the film edge, which goes beyond the
scope of the present paper. Still, within the considered restricted
class we may conclude, by \eqref{Eepslb} and \eqref{Eepsub}, that the
same result is true for the $z$-averages $\bar \phi_\eps$ of the
minimizers $\phi_\eps$ of $\mathcal E_\eps$ in the respective
class. The precise statement is in the following corollary.

\begin{corollary}
  \label{c:nolteno}
  Let $\rho > 0$, $0 < \lambda < \lambda_c$ and let $\mathcal E_\eps$
  be defined by \eqref{Eeps} with $\delta_\eps$ given by
  \eqref{deps}. Let
  $\phi_\eps \in H^1(\Omega^\eps) \cap L^\infty(\Omega^\eps)$ be a
  minimizer of $\mathcal E_\eps$ among all functions satisfying
  $\phi_\eps = 1$ in
  $\Omega^\eps \backslash (D_\rho \times (0, \eps \delta_\eps))$ and
  $\| \phi_\eps \|_{L^\infty(\Omega^\eps)} = 1$. Then if
  $\bar\phi_\eps$ is defined by \eqref{aveps}, we have
  $\bar\phi_\eps \to 1$ in $BV(D)$ as $\eps \to 0$.
\end{corollary}

We also point out that, despite asymptotic non-existence of
non-trivial minimizers of $E_\eps$ for $\lambda < \lambda_c$ the
effect of the dipolar interaction can still be seen in the energetics
via a renormalized line tension $\sigma = \sigma_0 - \lambda \sigma_1$
for the domain patterns in the plane. At the same time, Remark
\ref{r:nolte} also allows us to conclude that for
$\lambda > \lambda_c$ the minimizers in Corollary \ref{c:nolteno} must
develop spatial oscillations as $\eps \to 0$.

\begin{corollary}
  \label{c:noltesi}
  Let $\rho$, $\phi_\eps$ and $\bar\phi_\eps$ be as in Corollary
  \ref{c:nolteno}, and let $\lambda > \lambda_c$. Then
  $\bar\phi_\eps \not\to 1$ in $BV(D)$, as $\eps \to 0$.
\end{corollary}

\noindent In fact, it is possible to show that for
$\lambda > \lambda_c$ minimizers of $E_\eps$ or the $z$-averages of
minimizers of $\mathcal E_\eps$ cannot converge in $BV(D)$. Instead,
they develop fine oscillations throughout $D$ (for an analogous result
in micromagnetics, see \cite[Theorem 3.6]{kmn17}).

\section{Preliminaries}
\label{sec:preliminaries}

In this section, we collect a few basic facts for various terms
appearing both in the original energy in \eqref{E} and the reduced
energy in \eqref{EE}. In particular, we establish existence and
regularity of the minimizers of both energies. We remind the reader
that, except in the following lemma, we always consider a function
$\phi \in H^1(\Omega)$ to be extended by zero to the whole space
whenever we view $\phi$ as a function defined on $\R^3$. Similarly, a
function $\bar\phi \in H^1(D)$ is assumed to be extended by zero to
the rest of $\R^2$ whenever it is treated as a function defined on
$\R^2$.

We begin by a characterization of the non-local term appearing in
\eqref{E}. Recall that the derivative $\partial_z \phi$ in \eqref{E}
is understood in the distributional sense in the whole of $\R^3$.

\begin{lemma}
  \label{l:nloce}
  Let $\phi, \psi \in H^1(\Omega)$ and let $\tilde\phi, \tilde\psi$ be
  their extensions by zero to $\R^3 \backslash \Omega$,
  respectively. Then
  \begin{align}
    \label{dzphi3d}
    \int_{\R^3} \partial_z \tilde\phi  (-\Delta)^{-1} \partial_z
    \tilde\psi \, d^3 r := {1 \over 4 \pi} \int_{\R^3} \int_{\R^3} {\partial_z
    \tilde \phi(\mathbf r) \, \partial_z \tilde \psi(\mathbf r') \over
    |\mathbf r - \mathbf r'|} \, d^3 r \, d^3 r'
  \end{align}
  defines an inner product on $H^1(\Omega)$. Furthermore,
  $\int_{\R^3} \partial_z \tilde\phi (-\Delta)^{-1} \partial_z
  \tilde\psi \, d^3 r \leq \| \phi \|_{L^2(\Omega)} \| \psi
  \|_{L^2(\Omega)}$, and we have
  \begin{align}
    \label{dzphi3d2}
    \int_{\R^3} \partial_z \tilde\phi  (-\Delta)^{-1} \partial_z
    \tilde\psi \, d^3 r = -\int_{\Omega} \phi \, \partial^2_z
    (-\Delta)^{-1} \psi \, d^3 r, 
  \end{align}
  where $(-\Delta)^{-1} \psi \in W^{2,2}_{loc}(\R^3)$ is the Newtonian
  potential of $\tilde\psi$:
  \begin{align}
    \label{riesz}
    (-\Delta)^{-1} \psi(\mathbf r) := {1 \over 4 \pi} \int_\Omega
    {\psi(\mathbf r') \over |\mathbf r - \mathbf r'|} \, d^3 r \qquad
    \mathbf r \in \R^3. 
  \end{align}
\end{lemma}

\begin{proof}
  First of all, observe that since $\phi, \psi \in H^1(\Omega)$ and
  $\Omega$ is a bounded set with boundary of class $C^2$, we have
  $\tilde\phi, \tilde\psi \in BV(\R^3) \cap L^2(\R^3)$, with
  $\partial_z \tilde\phi = \partial_z \tilde\phi^a + \partial_z
  \tilde\phi^j$,
  where
  $\partial_z \tilde\phi^a = \mathcal L^3(\Omega)\llcorner \partial_z
  \phi$
  is the absolutely continuous part and
  $\partial_z \tilde\phi^j = \mathcal H^2(\partial \Omega) \llcorner
  (-\mathbf e_z \cdot \nu) T(\phi)$
  is the jump part \cite{evans}. Here, $T(\phi)$ denotes the trace of
  $\phi$ on $\partial \Omega$, $\nu$ is the outward unit normal vector
  to $\partial \Omega$ and $\mathbf e_z$ is the unit vector in the
  positive $z$ direction. Furthermore, since
  $T(\phi) \in L^2(\partial \Omega)$ by the trace embedding theorem
  \cite{evans}, it is easy to see that the right-hand side of
  \eqref{dzphi3d} defines an absolutely convergent integral. Then,
  arguing by approximation, we can write
  \begin{align}
    \label{dzphi3dFour}
    \int_{\R^3} \partial_z \tilde\phi  (-\Delta)^{-1} \partial_z
    \tilde\psi \, d^3 r = \int_{\R^3} {(\mathbf k \cdot \mathbf e_z)^2
    \over |\mathbf k|^2} \, \widehat \phi_{\mathbf k}^* \widehat\psi_{\mathbf k}
    \, {d^3 k \over (2 \pi)^3}, 
  \end{align}
  where $\widehat \phi_{\mathbf k}$ and $\widehat\psi_{\mathbf k}$ are the
  Fourier transforms of $\tilde\phi$ and $\tilde\psi$, respectively,
  with the convention
 \begin{align}
   \widehat \phi_{\mathbf k} := \int_{\R^3} e^{i \mathbf k \cdot \mathbf r}
   \tilde \phi(\mathbf r) \, d^3 r.
  \end{align}
  Thus, by Cauchy-Schwarz inequality and Parseval's identity, the
  first part of the statement follows. To complete the proof of the
  second part, we note that by standard elliptic regularity
  \cite{gilbarg}, we have
  $(-\Delta)^{-1} \psi \in W^{2,2}_{loc}(\R^3)$ and, therefore,
  $\partial_z^2 (-\Delta)^{-1} \psi \in L^2_{loc}(\R^3)$. The claim
  then follows by passing again to the Fourier space.
\end{proof}

As can be seen from the proof of Lemma \ref{l:nloce}, the Fourier
representation in \eqref{dzphi3dFour} of the integral in the
right-hand side of \eqref{dzphi3d} justifies our choice of notation
for the left-hand side of \eqref{dzphi3d}. Throughout the rest of the
paper, we drop the tildes from all the formulas involving the
extensions. An immediate corollary to Lemma \ref{l:nloce} is the
following, with the last statement obtained by testing the energy
against $\phi \equiv 1$.

\begin{corollary}
  \label{c:nloc3dpos}
  We have for all $\phi \in H^1(\Omega)$
  \begin{align}
    \label{nloc3dpos}
    0 \leq \int_{\R^3} \partial_z \phi  (-\Delta)^{-1} \partial_z
    \phi \, d^3 r \leq \int_\Omega \phi^2 \, d^3 r.
  \end{align}
  In particular,
  $\inf_{\phi \in H^1(\Omega)} \mathcal E(\phi) \leq 0$.
\end{corollary}

We next turn to existence and some basic properties of the minimizers
of $\mathcal E$. The arguments of the proof are fairly standard, based
on the direct method of calculus of variations and standard elliptic
regularity theory, with the exception of a separate treatment of the
contributions to the non-local term coming from the boundary trace of
$\phi$.

\begin{proposition}
  \label{p:exist3d}
  There exists a minimizer $\phi$ of $\mathcal E$ in \eqref{E} among
  all functions in $H^1(\Omega)$. Furthermore, we have
  $\phi \in C^\infty(\Omega) \cap C^{1,\alpha}(\overline \Omega)$ for
  all $\alpha \in (0, 1)$, and $\phi$ satisfies
  \begin{align}
    \label{EL3d}
    0 = \Delta \phi(\mathbf r) + (1 + \gamma) \phi(\mathbf r) -
    \phi^3(\mathbf r) 
    - {\gamma \over 4 \pi} \int_\Omega {\mathbf e_z
    \cdot (\mathbf r - \mathbf r') \over |\mathbf r - \mathbf
    r'|^3} \partial_z \phi(\mathbf r') d^3 r' \qquad \qquad \ \ \notag \\
    + {\gamma \over 4 \pi}
    \int_{\partial\Omega} {\mathbf e_z 
    \cdot (\mathbf r - \mathbf r') \over |\mathbf r - \mathbf
    r'|^3} (\mathbf e_z \cdot \nu(\mathbf r')) \phi(\mathbf r') d
    \mathcal H^2(\mathbf r') \qquad \forall \mathbf r \in \Omega,
  \end{align}
  where $\mathbf e_z$ is the unit vector in the positive $z$ direction
  and $\nu$ is the outward unit normal to $\partial \Omega$, with
  $\nu \cdot \nabla \bar \phi(\mathbf r) = 0$ for all
  $\mathbf r \in \partial \Omega$.
\end{proposition}

\begin{proof}
  By Lemma \ref{l:nloce} and Sobolev embedding \cite{evans}, the
  energy in \eqref{E} is well-defined and bounded below for all
  $\phi \in H^1(\Omega)$. Let $\phi_n \in H^1(\Omega)$ be a minimizing
  sequence. Then by Corollary \ref{c:nloc3dpos} and Cauchy-Schwarz
  inequality we have
  \begin{align}
    \frac12 \| \nabla \phi_n \|_{L^2(\Omega)}^2 - \frac12 (1 + \gamma)
    |\Omega|^{1/2} \| \phi_n \|_{L^4(\Omega)}^2 + \frac14 \| \phi_n 
    \|_{L^4(\Omega)}^4 \leq C,
  \end{align}
  for some $C > 0$ independent of $n$. Therefore, upon extraction of a
  subsequence we may assume that $\phi_n \rightharpoonup \phi$ in
  $H^1(\Omega)$ as $n \to \infty$, and upon further extraction we also
  have $\phi_n \to \phi$ in $L^p(\Omega)$ for all $1 \leq p < 6$
  \cite{evans}. In particular, up to a subsequence (not relabeled) we
  have $\phi_n \to \phi$ in $L^2(\Omega)$, and by Lemma \ref{l:nloce}
  we also have
  $\int_{\R^3} \partial_z \phi_n (-\Delta)^{-1} \partial_z \phi_n \,
  d^3 r \to \int_{\R^3} \partial_z \phi (-\Delta)^{-1} \partial_z \phi
  \, d^3 r$ as $n \to \infty$. Therefore, by lower semicontinuity of
  the gradient squared term in the energy, we have $\liminf_{n \to
    \infty} \mathcal E(\phi_n) \geq \mathcal E(\phi)$, and so $\phi$
  is a minimizer. 

  By Lemma \ref{l:nloce} and an explicit calculation, the energy in
  \eqref{E} is Fr\'echet differentiable, and the minimizer $\phi$
  satisfies
  \begin{align}
    \int_\Omega \left( \nabla \phi \cdot \nabla \psi - (1 + \gamma)
    \phi \psi + \phi^3 \psi \right) d^3 r + \gamma
    \int_{\R^3} \partial_z \psi (-\Delta)^{-1} \partial_z \phi \, d^3
    r = 0,
  \end{align}
  for every $\psi \in H^2(\Omega)$ extended by zero to the whole of
  $\R^3$. Therefore, by Lemma \ref{l:nloce} we have
  \begin{align}
    \label{EL3dstrong}
    0 = \int_{\partial \Omega} \phi \nabla \psi \cdot \nu \, d
    \mathcal H^2 - \int_\Omega \phi \Delta \psi \, d^3 r  
    - \int_\Omega \left( (1 + \gamma) \phi - \phi^3 +
    \gamma \partial_z^2 (-\Delta)^{-1} \phi \right) \psi \, d^3 r, 
  \end{align}
  where the boundary integral is evaluated on traces of
  $\phi \in H^1(\Omega)$ and $\nabla \psi \in H^1(\Omega; \R^3)$.
  Since the bracket in the last integral in \eqref{EL3dstrong} is in
  $L^2(\Omega)$ by Sobolev embedding and Lemma \ref{l:nloce}, by
  standard elliptic estimates \cite{agmon59,lunardi} we have
  $\phi \in W^{2,2}(\Omega)$ and, therefore, $\phi \in L^p(\Omega)$
  for any $1 \leq p < \infty$, again, by Sobolev embedding (recall
  that $\Omega \subset \R^3$ is a bounded open set with boundary of
  class $C^2$). Then, again by standard elliptic regularity we also
  have $(-\Delta)^{-1} \phi \in W^{2,p}(\Omega)$ and, hence,
  $\partial_z^2 (-\Delta)^{-1} \phi \in L^p(\Omega)$. Thus, we
  conclude that $\phi \in W^{2,p}(\Omega)$ as well, and by Sobolev
  embedding $\phi \in C^{1,\alpha}(\overline\Omega)$, for any
  $\alpha \in (0,1)$. In particular, $\phi$ satisfies Neumann boundary
  condition.

  Finally, to arrive at \eqref{EL3d} we note that with the above
  regularity of $\phi$ we can write
  \begin{align}
    \label{int3d}
    \partial_z^2 (-\Delta)^{-1} \phi(\mathbf r) 
    & = -{1 \over 4 \pi}
      \int_\Omega {\mathbf e_z \cdot 
      (\mathbf r - \mathbf r') \over |\mathbf r -\mathbf r'|^3}
    \partial_z \phi(\mathbf r') \,
    d^3 r' \notag \\
    & +  {1 \over 4 \pi} \int_{\partial
    \Omega}  {\mathbf e_z \cdot
    (\mathbf r - \mathbf r') \over |\mathbf r -\mathbf r'|^3} (\mathbf
    e_z \cdot \nu(\mathbf r'))
    \phi(\mathbf r') \, d \mathcal H^2(\mathbf r'),
  \end{align}
  in $\mathcal D'(\R^3)$. The last term in the right-hand side of
  \eqref{int3d} defines a smooth function of $\mathbf r \in \Omega$,
  while the first term has derivatives belonging to $L^p(\Omega)$, in
  view of the fact that $\phi \in W^{1,p}(\Omega)$ and using standard
  elliptic regularity. Thus, we can apply a bootstrap argument to
  establish interior $C^\infty$ regularity of $\phi$ in $\Omega$. This
  then allows us to obtain \eqref{EL3d} from \eqref{EL3dstrong}.
\end{proof}

\begin{remark}
  \label{r:Linfty3d}
  Observe that by Proposition \ref{p:exist3d} every minimizer $\phi$
  of $\mathcal E$ over $H^1(\Omega)$ is bounded. However, it is not a
  priori clear under which conditions the $L^\infty$ norm of $\phi$
  remains bounded as $\delta \to 0$, with other parameters such as
  $\gamma$ or the diameter of $\Omega$ possibly going to infinity. It
  is natural to expect that in some thin film regimes the minimizers
  may develop a boundary layer near the edge, i.e., in the vicinity of
  $\Omega \backslash (D \times (0, \delta))$, and blow up at
  $\partial D \times (0, \delta)$ as $\delta \to 0$.
\end{remark}

Before turning to the discussion of the reduced energy $E$ in
\eqref{EE}, we consider the contribution of the film's edge to the
non-local term in the energy in \eqref{E}.  We have the following
estimate for the contribution of the edge to the non-local term in the
energy in the following lemma. Note that this estimate is expected to
be optimal for small $\delta$, since for $\phi = 1$, for example, the
self-interaction energy associated with the edge can be easily seen to
be of order $\delta^2$.

\begin{lemma}
  \label{l:edge}
  Let $\phi \in H^1(\Omega)$ and $\delta > 0$. Then 
 \begin{align}
    \label{L2edge3drough}
   \left| \int_{\R^3} \partial_z \phi (-\Delta)^{-1} \partial_z \phi
   \, d^3 r - \int_{\R^3} \partial_z (\chi_\delta \phi)
   (-\Delta)^{-1} \partial_z (\chi_\delta \phi)
   \, d^3 r \right| \leq 3 \| \phi \|_{L^2(\Omega)} \| \phi
   \|_{L^2(\Omega \backslash \Omega_{2\delta})}.   
  \end{align}
  Furthermore, there exists $\delta_0 > 0$ depending only on $D$ such
  that for all $0 < \delta \leq \delta_0$ we have for all $\phi \in
  H^1(\Omega) \cap L^\infty(\Omega)$: 
  \begin{align}
    \label{L2edge3d}
    \left| \int_{\R^3} \partial_z \phi (-\Delta)^{-1} \partial_z \phi
    \, d^3 r - \int_{\R^3} \partial_z (\chi_\delta \phi)
    (-\Delta)^{-1} \partial_z (\chi_\delta \phi)
    \, d^3 r \right| \notag \\
    \leq 98 |\partial D| \delta^2 \| \phi
    \|_{L^\infty(\Omega)}^2.  
  \end{align}
\end{lemma}

\begin{proof}
  Denoting the left-hand side of \eqref{L2edge3d} by $R$, writing
  $\phi = \chi_\delta \phi + (1 - \chi_\delta) \phi$ and expanding the
  difference, we have $R \leq 2 R_1 + R_2$, where
  \begin{align}
    R_1 & := \left| \int_{\R^3} \partial_z ((1 - \chi_\delta) \phi)
          (-\Delta)^{-1} \partial_z (\chi_\delta \phi)
          \, d^3 r \right|, \\
    R_2 & := \int_{\R^3} \partial_z ((1 - \chi_\delta) \phi)
          (-\Delta)^{-1} \partial_z ((1 - \chi_\delta) \phi)
          \, d^3 r.
  \end{align}
  The rough bound in \eqref{L2edge3drough} is then an immediate
  consequence of Lemma \ref{l:nloce}.

  To proceed towards the proof of \eqref{L2edge3d}, we still estimate
  $R_2$ roughly:
  \begin{align}
    \label{R2rough}
    R_2 \leq \| \phi \|_{L^2(\Omega \backslash \Omega_{6\delta})}^2
    \leq 14 |\partial D| \delta^2 \| \phi \|_{L^\infty(\Omega)}^2, 
  \end{align}
  where we chose $\delta$ so small depending only on $D$ that
  $|(D + B_\delta) \backslash D_{6\delta}| \leq 14 |\partial D|
  \delta$
  and, hence,
  $|\Omega \backslash \Omega_{6\delta}| \leq 14 |\partial D|
  \delta^2$.
  Focusing on $R_1$, we write, using again Lemma \ref{l:nloce} to
  estimate the first term:
  \begin{align}
    R_1 & \leq \left| \int_{\R^3} \partial_z ((1 - \chi_\delta) \phi)
          (-\Delta)^{-1} \partial_z ((\chi_\delta - \chi_{3
          \delta}) \phi)
          \, d^3 r \right| \notag \\
        & \quad + \left| \int_{\R^3} \partial_z ((1 - \chi_\delta) \phi)
          (-\Delta)^{-1} \partial_z (\chi_{3\delta} \phi)
          \, d^3 r \right| \notag \\
        & \leq 14  |\partial D| \delta^2 \| \phi \|_{L^\infty(\Omega)}^2
          \notag \\
        & \quad + {1 \over 4 \pi} \left| \int_{\Omega \backslash
          \Omega_{2\delta}} \int_{\Omega_{3 \delta}} {3 (\mathbf e_z
          \cdot (\mathbf r - \mathbf r'))^2 -  | \mathbf r -
          \mathbf r' |^2 \over | \mathbf r -
          \mathbf r' |^5} (1 -
          \chi_\delta(\mathbf r)) \chi_{3 \delta} (\mathbf r')
          \phi(\mathbf r) \, \phi(\mathbf r') \, d^3 r' \, d^3 r \right|
          \notag \\ 
        & \leq 14 |\partial D| \delta^2 \| \phi \|_{L^\infty(\Omega)}^2
          +  {1 \over \pi} \| \phi \|_{L^\infty(\Omega)}^2
          \int_{\Omega \backslash \Omega_{2\delta}} \left( 
          \int_{\Omega_{3 \delta}} {1 \over | \mathbf r -
          \mathbf r' |^3} \, d^3 r' \right) d^3 r \notag \\
        & \leq  14 |\partial D| \delta^2 \| \phi \|_{L^\infty(\Omega)}^2
          +  {1 \over \pi} \delta^2 \| \phi \|_{L^\infty(\Omega)}^2
          \int_{(D + B_\delta) \backslash D_{2\delta}} \left( 
          \int_{\R^2 \backslash B_\delta(\mathbf r)} {1 \over | \mathbf r -
          \mathbf r' |^3} \, d^2 r' \right) d^2 r \notag \\
        & \leq 42  |\partial D| \delta^2 \| \phi
          \|_{L^\infty(\Omega)}^2. 
  \end{align}
  Combining this estimate with \eqref{R2rough} yields \eqref{L2edge3d}.
\end{proof}

\noindent Let us note that the universal constants appearing in Lemma
\ref{l:edge} are not intended to be optimal.

We now proceed to establishing existence and regularity of the
minimizers of $E$ from \eqref{EE} among all $\bar \phi \in H^1(D)$.

\begin{proposition}
  \label{p:exist2d} 
  For every $\alpha > 0$ and every $\delta > 0$ such that
  $\alpha \delta^2 < 1$ there exists a minimizer $\bar \phi$ of $E$ in
  \eqref{EE} among all functions in $H^1(D)$. Furthermore, we have
  $\bar \phi \in C^\infty(D) \cap C^{1,\alpha}(\overline D)$ for all
  $\alpha \in (0,1)$, and $\bar \phi$ satisfies for every $\mathbf r
  \in D$: 
  \begin{align}
    \label{EL2d}
    0 & = (1 - \alpha \delta^2) \Delta \bar \phi(\mathbf r) + \bar
        \phi(\mathbf r) - \bar \phi^3(\mathbf r) \notag \\ 
      & \qquad + {\gamma \delta \over 8 \pi} \chi_\delta(\mathbf
        r) \int_{\R^2} { 2 \chi_\delta(\mathbf
        r) \bar\phi(\mathbf r) - \chi_\delta(\mathbf
        r - \mathbf z) \bar\phi(\mathbf r - \mathbf z) - \chi_\delta(\mathbf
        r + \mathbf z) \bar\phi(\mathbf r + \mathbf z) \over |\mathbf
        z|^3} \, d^2 z,
  \end{align}
  with $\nu \cdot \nabla \bar \phi(\mathbf r) = 0$ for all
  $\mathbf r \in \partial D$, where $\nu$ is the outward unit
  normal.
\end{proposition}

\begin{proof}
  The proof is analogous to that of Proposition \ref{p:exist2d} and is
  simpler, because now $\chi_\delta \bar\phi \in H^1(\R^2)$. This
  means that by interpolation the non-local term in the energy may be
  controlled by the $H^1(\R^2)$ norm of $\chi_\delta \bar\phi$
  \cite{lieb-loss}, which, in turn, can be controlled by the $H^1(D)$
  norm of $\phi$. Thus, if $\bar \phi_n \in H^1(D)$ is a minimizing
  sequence, we may write
  \begin{align}
    C \geq \frac12 (1 - \alpha \delta^2) \| \nabla \bar\phi_n
    \|_{L^2(D)}^2 - \frac12  \| \bar\phi_n \|_{L^2(D)}^2 + \frac14  \|
    \bar\phi_n \|_{L^4(D)}^4 - {\gamma \delta \over 4}  \|
    \chi_\delta \bar\phi_n \|_{\mathring{H}^{1/2}(\R^2)}^2 \notag \\
    \geq \frac12 (1 - \alpha \delta^2) \| \nabla \bar\phi_n
    \|_{L^2(D)}^2 - 
    \frac12  \| \bar\phi_n \|_{L^2(D)}^2 + \frac14  \| \bar\phi_n
    \|_{L^4(D)}^4 \notag \\ 
    - {\gamma \delta \over 4} \| \chi_\delta \bar \phi_n \|_{L^2(\R^2)}
    \| \nabla (\chi_\delta \bar\phi_n) \|_{L^2(\R^2)} \notag \\
    \geq \frac12 (1 - \alpha \delta^2) \| \nabla \bar\phi_n
    \|_{L^2(D)}^2 - 
    \frac12  \| \bar\phi_n \|_{L^2(D)}^2 + \frac14  \| \bar\phi_n
    \|_{L^4(D)}^4 \notag \\ 
    -{\gamma \delta \over 4} \| \bar \phi_n \|_{L^2(D)} \left(   \| \nabla
    \bar \phi_n \|_{L^2(D)} + \| \nabla \chi_\delta
    \|_{L^\infty(\R^2)} \| \bar \phi_n \|_{L^2(D \backslash
    D_{2\delta})} \right) \notag \\
    \geq \frac12 (1 - \alpha \delta^2) \| \nabla \bar\phi_n
    \|_{L^2(D)}^2 - 
    \frac12  \| \bar\phi_n \|_{L^2(D)}^2 + \frac14  \| \bar\phi_n
    \|_{L^4(D)}^4 \notag \\ 
    -{\gamma \over 4} \| \bar \phi_n \|_{L^2(D)} \left(   \delta \| \nabla
    \bar \phi_n \|_{L^2(D)} + 2 \| \bar \phi_n \|_{L^2(D \backslash
    D_{2\delta})} \right), \label{2dinterpol}
  \end{align}
  for some $C > 0$ independent of $n$ and all $\delta$ sufficiently
  small. Therefore, by Cauchy-Schwarz and Young's inequalities we
  obtain
  \begin{align}
    \| \nabla \bar\phi_n \|_{L^2(D)}^2 - C_1 \| \bar\phi_n
    \|_{L^4(D)}^2 + C_2 \| \bar\phi_n \|_{L^4(D)}^4 \leq C_3,
  \end{align}
  for some $C_1, C_2, C_3 > 0$ independent of $n$. This yields
  compactness in $H^1(D)$ which, upon extraction of a subsequence,
  produces $\bar\phi \in H^1(D)$ such that
  $\bar\phi_n \rightharpoonup \bar\phi$ in $H^1(D)$,
  $\bar\phi_n \to \bar\phi$ in $L^p(D)$ for any $1 \leq p < \infty$
  (recall that $D \subset \R^2$ and is bounded), and again by
  interpolation we have $\bar\phi_n \to \bar\phi$ in $H^{1/2}(\R^2)$
  \cite{lieb-loss}. Finally, by lower semicontinuity of the gradient
  squared term, we obtain that $\bar\phi$ is a minimizer.

  Once existence of a minimizer $\bar\phi$ is established, the weak
  form of \eqref{EL2d} is obtained by an explicit computation:
  \begin{align}
    0 & = (1 - \alpha \delta^2) \int_D \nabla \bar \phi \cdot \nabla \bar
        \psi \, d^2 r + \int_D (\bar\phi^3 - \bar\phi) \bar \psi \, d^2 r
        \notag \\
      & - {\gamma \delta \over 8 \pi} \int_{\R^2} \int_{\R^2}
        {(\chi_\delta(\mathbf 
        r) \bar\phi(\mathbf r) - \chi_\delta(\mathbf
        r') \bar\phi(\mathbf 
        r')) (\chi_\delta(\mathbf
        r) \bar\psi(\mathbf r) - \chi_\delta(\mathbf
        r') \bar\psi(\mathbf 
        r')) \over |\mathbf r - \mathbf r'|^3} \, d^2 r \, d^2 r',
  \end{align}
  for any $\bar\psi \in H^1(D)$. Passing to Fourier space in the last
  term, we can then interpret this equation distributionally in $D$:
  \begin{align}
    \label{EL2ddistr}
    0 = (1 - \alpha \delta^2) \Delta \bar \phi + \bar\phi -
    \bar\phi^3 + {\gamma \delta \over 2} \chi_\delta (-\Delta)^{1/2}
    (\chi_\delta \bar\phi),
  \end{align}
  where for test functions the operator $ (-\Delta)^{1/2}$ is defined
  by \eqref{halflapl} (for a more detailed discussion of various
  representations of half-Laplacian in $\R^2$, see
  \cite{lmm:jns15}). Moreover, since
  $\chi_\delta \bar\phi \in H^1(D)$, the last term in
  \eqref{EL2ddistr} belongs to $L^2(\R^2)$, and by standard elliptic
  regularity $\bar \phi \in H^2(D)$, with Neumann boundary
  condition. Applying bootstrap then yields the remaining claims.
\end{proof}

We finish this section with an estimate for the energy $E$ on a fixed
domain $D$ and small $\delta$ that will be useful in
establishing the asymptotic behavior of the energy for $\delta \to 0$.  

\begin{lemma}
  \label{l:2dsmdel}
  There exists $\delta_0 > 0$ depending only on $\alpha$ and $D$ such
  that for all $0 < \delta \leq \delta_0$ and all
  $\bar\phi \in H^1(D)$ there holds
  \begin{align}
    E(\bar\phi) \geq \frac18 \| \nabla \bar\phi \|_{L^2(D)}^2 -
    \frac12 \left( 1 + {\gamma^2 \delta^2 \over 4} \right) \| \bar\phi
    \|_{L^2(D)}^2  + \frac14 \| \bar\phi \|_{L^4(D)}^4  + \frac14 |D| 
    \notag \\ -
    \gamma |\partial D|^{1/4} \delta^{1/4} \|
    \bar\phi \|_{L^2(D)} \| \bar\phi \|_{L^4(D)}.
  \end{align}
\end{lemma}

\begin{proof}
  We argue as in the proof of Proposition \ref{p:exist2d}. Taking
  $\alpha \delta^2 \leq \frac12$ and using the estimate in
  \eqref{2dinterpol}, with the help of Young's inequality we obtain 
  \begin{align}
    E(\bar \phi) \geq \frac14 \| \nabla \bar\phi \|_{L^2(D)}^2 - 
    \frac12  \| \bar\phi \|_{L^2(D)}^2 + \frac14  \| \bar\phi
    \|_{L^4(D)}^4 + \frac14 |D| \notag \\ 
    -{\gamma \over 4} \| \bar \phi \|_{L^2(D)} \left(   \delta \| \nabla
    \bar \phi \|_{L^2(D)} + 2 \| \bar \phi \|_{L^2(D \backslash
    D_{2\delta})} \right) \notag \\
    \geq \frac18 \| \nabla \bar\phi \|_{L^2(D)}^2 - 
    \frac12  \| \bar\phi \|_{L^2(D)}^2 + \frac14  \| \bar\phi
    \|_{L^4(D)}^4 + \frac14 |D| \notag \\ 
    -{\gamma^2 \delta^2 \over 8} \| \bar \phi \|_{L^2(D)}^2 - {\gamma
    \over 2} \| \bar \phi \|_{L^2(D)} \| \bar \phi \|_{L^2(D 
    \backslash D_{2\delta})}. \label{y14}
  \end{align}
  On the other hand, choosing $\delta$ so small that
  $|D \backslash D_{2\delta}| \leq 16 |\partial D| \delta$, by
  Cauchy-Schwarz inequality we have
  \begin{align}
    \label{cs14}
    \| \bar \phi \|_{L^2(D 
    \backslash D_{2\delta})} \leq 2 |\partial D|^{1/4} \delta^{1/4} \|
    \bar \phi \|_{L^4(D)}. 
  \end{align}
  Combining \eqref{cs14} with \eqref{y14} then yields the result.
\end{proof}

\section{Proof of Theorem \ref{t:main}}
\label{sec:proof1}

The main ingredient in the proof of Theorem \ref{t:main} is a careful
estimate of the non-local part of the three-dimensional energy
$\mathcal E$ evaluated on $\chi_\delta \phi$ (to exclude the effect of
the edge) in terms of the non-local part of the two-dimensional energy
$E$ evaluated on $\chi_\delta \bar\phi$ , where $\bar\phi$ is given by
\eqref{av}. The key point is that the difference between the two can
be controlled by the gradient squared term in $\mathcal E(\phi)$. Note
that a similar argument in the periodic setting was recently
introduced in \cite{kmn17}. We establish the estimate in the following
lemma.

\begin{lemma}
  \label{l:ded}
  Let $\phi \in H^1(\Omega)$ be extended by zero to the whole of
  $\R^3$ and let $\bar\phi$ be defined by \eqref{av}. Then
  \begin{align}
    \label{keyest}
    \left| \int_{\R^3} \left( \partial_z (\chi_\delta \phi)
    (-\Delta)^{-1} \partial_z (\chi_\delta \phi) -
    \chi_\delta^2 \phi^2
    \right) d^3 r + {\delta^2 \over 2} \int_{\R^2} 
    \chi_\delta \bar \phi (-\Delta)^{1/2} \chi_\delta \bar \phi \, d^2
    r \right| \notag \\
    \leq {\delta^2 \over 2} \int_\Omega |\nabla (\chi_\delta
    \phi)|^2 \, d^3 r.
  \end{align}
\end{lemma}

\begin{proof}
  To simplify the notations, let us introduce
  $\psi := \chi_\delta \phi$ and $\bar\psi := \chi_\delta
  \bar\phi$.
  Notice that in view of Lemma \ref{l:nloce}, we can argue by
  approximation and assume that $\psi \in C^\infty_c(\R^3)$ and
  $\bar \psi \in C^\infty_c(\R^2)$.  Next, for each $z \in \R$ define
  the Fourier transform of $\psi = \psi(x, y, z)$ in the first two
  variables $\mathbf (x, y) = \mathbf r \in \R^2$:
  \begin{align}
    \widehat\psi_\mathbf{k}(z) := \int_{\R^2} e^{i \mathbf k \cdot \mathbf
    r} \psi(\mathbf r, z) \, d^2 r \qquad \mathbf k \in \R^2.
  \end{align}

  We write the three-dimensional dipolar interaction energy (up to a
  factor) in terms of the associated potential
  $\varphi \in C^\infty(\R^3)$:
  \begin{align}
    \label{eq:7}
    \mathcal E_d(\psi) := \int_{\R^3} \partial_z \psi
    (-\Delta)^{-1} \partial_z  \psi \, d^3 r = -\int_{\R^3} \varphi
    \, \partial_z \psi \, d^3 r, \qquad \qquad \Delta \varphi
    = \partial_z \psi \quad \text{in} \quad \R^3. 
  \end{align}
  Passing to the Fourier space, with the help of Parseval's identity
  we get
  \begin{align}
    \label{eq:8}
    \mathcal E_d(\psi) = - {1 \over (2 \pi)^2} \int_{\R^2}
    \int_0^\delta \widehat 
    \varphi_k^* \partial_z \widehat \psi_\mathbf{k} \, dz \, d^2 k,
  \end{align}
  where we introduced the Fourier transform
  $\widehat\varphi_\mathbf{k} = \widehat\varphi_\mathbf{k} (z)$ of
  $\varphi$, which solves
  \begin{align}
    \label{eq:9}
    {d^2 \widehat\varphi_\mathbf{k}  \over dz^2} - |\mathbf k|^2
    \widehat\varphi_\mathbf{k}  = \partial_z \widehat\psi_\mathbf{k}
    \qquad z  \in \R. 
  \end{align}

  Introducing the fundamental solution
  \begin{align}
    \label{eq:10}
    H_\mathbf{k} (z) := {e^{-|\mathbf{k}| |z|} \over |\mathbf{k}|} \qquad
    \mathbf{k} \in \R^2 
    \backslash \{0\}, \ z \in \R,
  \end{align}
  of the ordinary differential equation
  \begin{align}
    \label{eq:17}
    - {d^2 H_\mathbf{k} (z) \over dz^2} + |\mathbf{k}|^2 H_\mathbf{k}
    (z) = 2 \delta^{(1)}(z) \qquad z \in \R, 
  \end{align}
  where $\delta^{(1)}(z)$ is the one-dimensional Dirac delta-function,
  we can write the solution of \eqref{eq:9} in terms of
  $H_\mathbf{k}(z)$ as
  \begin{align}
    \label{eq:11}
    \widehat\varphi_\mathbf{k} (z) = - \frac12 \int_\R H_\mathbf{k} (z
    - z') \partial_z \widehat\psi_\mathbf{k} (z') dz' \qquad z \in
    \R. 
  \end{align}
  Thus, we have
  \begin{align}
    \label{eq:12}
    \mathcal E_d(\psi) = \frac{1}{8 \pi^2} \int_{\R^2}\int_\R \int_\R
    \partial_z \widehat\psi_\mathbf{k}^* (z) H_\mathbf{k}(z -
    z') \partial_z \widehat\psi_\mathbf{k} (z') dz \, dz' \, d^2 k.
  \end{align}

  Introduce now $H_\mathbf{k}^{(0)}(z - z') := |\mathbf k|^{-1}$ and
  observe that
  \begin{multline}
    \label{eq:13}
    \mathcal E_d^{(0)}(\psi) := \frac{1}{8 \pi^2} \int_{\R^2} \int_\R
    \int_\R
    \partial_z \widehat\psi_\mathbf{k}^* (z) H_k^{(0)}(z - z')
    \partial_z \widehat\psi_\mathbf{k} (z') dz \, dz' \, d^2 k \\
    = \frac{1}{8 \pi^2} \int_{\R^2} {1 \over |\mathbf k|} \int_\R
    \int_\R
    \partial_z \widehat\psi_\mathbf{k}^* (z)
    \partial_z \widehat\psi_\mathbf{k} (z') dz \, dz' \, d^2 k = 0.
  \end{multline}
  Similarly, with $H_\mathbf{k}^{(1)}(z - z') := -|z - z'|$ we have by
  Parseval's identity
  \begin{multline}
    \label{eq:14}
    \mathcal E_d^{(1)}(\psi) = \frac{1}{8 \pi^2} \int_{\R^2} \int_\R
    \int_\R \partial_z \widehat\psi_\mathbf{k}^* (z)
    H_\mathbf{k}^{(1)}(z - z')
    \partial_z \widehat\psi_\mathbf{k} (z') dz \, dz' \, d^2 k \\
    = \frac{1}{(2 \pi)^2} \int_{\R^2} \int_\R \partial_z
    \widehat\psi_\mathbf{k}^* (-\partial_z^2)^{-1} \partial_z
    \widehat\psi_\mathbf{k} dz \\
    = \frac{1}{(2 \pi)^2} \int_{\R^2} \int_\R \left|
      \widehat\psi_\mathbf{k} \right|^2 dz \, d^2 k = \int_{\R^3}
    \psi^2 d^3 r.
  \end{multline}
  In turn, with $H_k^{(2)}(z - z') = \tfrac12 |k| (z - z')^2$ we have
  \begin{multline}
    \label{eq:19}
    \mathcal E_d^{(2)}(\psi) = \frac{1}{8 \pi^2} \int_{\R^2} \int_\R
    \int_\R \partial_z \widehat\psi_\mathbf{k}^* (z)
    H_\mathbf{k}^{(2)}(z - z')
    \partial_z \widehat\psi_\mathbf{k} (z') dz \, dz' \, d^2 k \\
    = -\frac{1}{8 \pi^2} \int_{\R^2} |\mathbf k| \int_0^\delta
    \int_0^\delta \widehat\psi_\mathbf{k}^* (z)
    \widehat\psi_\mathbf{k} (z') \, dz \, dz' \, d^2 k \\
    = - \frac{\delta^2}{8 \pi^2} \int_{\R^2} |\mathbf k| \left|
      \widehat{\overline \psi}_\mathbf{k} \right|^2 d^2 k = -
    {\delta^2 \over 2} \int_{\R^2} \bar\psi \, (-\Delta)^{1/2} \,
    \bar\psi \, d^2 r.
  \end{multline}

  We now estimate the energy difference
  $\Delta \mathcal E_d(\psi) = \mathcal E_d(\psi) - \mathcal
  E_d^{(0)}(\psi) - \mathcal E_d^{(1)}(\psi) - \mathcal
  E_d^{(2)}(\psi)$.
  Introduce
  $I_\mathbf{k} (z) := H_\mathbf{k} (z) - H_\mathbf{k}^{(0)}(z) -
  H_\mathbf{k}^{(1)}(z) - H_\mathbf{k}^{(2)}(z)$,
  and observe that $I_\mathbf{k} \in C^2(\R)$ and
  \begin{align}
    I''_\mathbf{k}(z) = {d^2 I_\mathbf{k} (z) \over dz^2} =
    -|\mathbf{k}| \left( 1 - e^{-|\mathbf{k}| 
    |z|}\right), 
  \end{align}
  In particular, we have
  \begin{align}
    \label{eq:38}
    |I_\mathbf{k}''(z)| \leq |k|^2 \delta \qquad \forall |z| \leq \delta. 
  \end{align}
  Integrating by parts, we express the excess energy in terms of
  $I_\mathbf{k}''(z)$:
  \begin{multline}
    \label{eq:27}
    \Delta \mathcal E_d(\psi) = \frac{1}{8 \pi^2} \int_{\R^2} \int_\R
    \int_\R
    \partial_z \widehat\psi_\mathbf{k}^* (z) I_\mathbf{k}(z - z')
    \partial_z \widehat\psi_\mathbf{k} (z') dz \, dz' \\
    = -\frac{1}{8 \pi^2} \int_{\R^2} \int_\R \int_\R
    \widehat\psi_\mathbf{k}^* (z) I_\mathbf{k}''(z - z')
    \widehat\psi_\mathbf{k} (z') \, dz \, dz' \, d^2 k.
  \end{multline}
  Therefore, applying Cauchy-Schwarz inequality and using
  \eqref{eq:38}, we obtain
  \begin{align}
  \label{eq:39}
    |\Delta \mathcal E_d(\psi)| 
    & \leq \frac{1}{8 \pi^2} \int_{\R^2} \int_0^\delta \int_0^\delta
      \left| \widehat\psi_\mathbf{k} (z) \right| \, \left| I_\mathbf{k}''(z -
      z') \right| \, \left|
      \widehat\psi_\mathbf{k} (z') \right| \, dz \, dz' \, d^2 k \notag \\
    & \leq  \frac{\delta^2}{8 \pi^2} \int_{\R^2} \int_0^\delta |\mathbf k|^2  
      \left| \widehat\psi_\mathbf{k} (z) \right|^2 \, dz \, \, d^2
      k \leq  {\delta^2 \over 2} 
      \int_\Omega
      |\nabla \psi|^2 d^3 r,
  \end{align}
  which yields the claim. 
\end{proof}

\begin{proof}[Proof of Theorem \ref{t:main}]
  We begin with a lower bound and split the energy into the local and
  the dipolar parts:
  \begin{align}
    \mathcal E(\phi) = \mathcal E_l(\phi) + {\gamma \over 2} \mathcal
    E_d(\phi), 
  \end{align}
  where $\mathcal E_d$ is defined in \eqref{eq:7}. Applying Jensen's
  inequality to the positive terms, for any $\alpha_1 > 0$,
  $\alpha_2 > 0$ and $\delta > 0$ such that
  $(\alpha_1 + \alpha_2 \gamma) \delta^2 < 1$ we have
  \begin{align}
    \mathcal E_l(\phi) 
    & = \int_\Omega \left( \frac12 | \nabla \phi|^2
      - \frac12 \phi^2 + \frac14 \phi^4 - \frac14 \right) d^3 r \notag
    \\
    & \geq \int_{D \times (0, \delta)} \left( \frac12 (1 - \alpha_2
      \gamma \delta^2) | 
      \nabla \bar \phi|^2 - \frac12 \phi^2 + \frac14 \bar \phi^4 -
      \frac14 \right) d^3 r + {\alpha_2 \gamma \delta^2 \over 2}
      \int_\Omega |\nabla \phi|^2 d^3 r
      \notag \\
    & \geq \int_{D \times (0, \delta)} \left( \frac12  (1 - \alpha_2
      \gamma \delta^2)| 
      \nabla \bar \phi|^2 
      + \frac14 \left( 1 - \bar \phi^2 \right)^2 \right) d^3 r \notag
    \\ 
    & \qquad \qquad \qquad -
      \frac12 \int_{D \times (0, \delta)} (\phi - \bar \phi)^2 d^3 r + {\alpha_2
      \gamma \delta^2 \over 2}  \int_\Omega |\nabla \phi|^2 d^3 r. 
  \end{align}
  Therefore, by Poincar\'e's inequality we obtain
  \begin{align}
    \label{Epoinc}
    \mathcal E_l(\phi) \geq \int_{D \times (0, \delta)} \left( \frac12
    (1 - \alpha_1 \delta^2 - \alpha_2 \gamma \delta^2) | \nabla \bar
    \phi|^2 + \frac14 \left( 1 - \bar \phi^2 \right)^2 \right) d^3 r +
    {\alpha_2 \gamma \delta^2 \over 2}  \int_\Omega |\nabla \phi|^2
    d^3 r, 
  \end{align}
  with $\alpha_1 = \pi^{-2}$.

  Turning now to the dipolar part, we observe that by Lemma
  \ref{l:edge} we have for all $\delta$ sufficiently small:
  \begin{align}
    \label{Edbeta}
    \mathcal E_d(\phi)  \geq \mathcal E_d(\chi_\delta \phi) - 98
    \delta^2 \| \phi \|_{L^\infty(\Omega)}^2
    |\partial D|.
  \end{align}
  At the same time, by Lemma \ref{l:ded} we may write
  \begin{align}
    \label{Edlb}
    \mathcal E_d(\chi_\delta \phi) \geq \int_\Omega \chi_\delta^2
    \phi^2 \, d^3 r - {\delta^2 \over 2} \int_{\R^2}
    \chi_\delta \bar\phi (-\Delta)^{1/2} \chi_\delta \bar \phi \, d^2 r -
    {\delta^2 \over 2} \int_\Omega |\nabla (\chi_\delta \phi)|^2
    d^3 r.
  \end{align}
  By Young's inequality, the last term in \eqref{Edlb} may be
  estimated as
  \begin{align}
    {\delta^2 \over 2} \int_\Omega | \nabla (\chi_\delta \phi)|^2
    d^3 r \leq \delta^2 \int_\Omega \left( | \nabla \chi_\delta
    |^2 \phi^2 + \chi_\delta^2 |\nabla \phi|^2 \right) d^3
    r \notag \\
    \leq 4 \int_{\Omega \backslash \Omega_{2\delta}} \phi^2 d^3 r +
    \delta^2 \int_\Omega |\nabla \phi|^2 d^3 r.
  \end{align}
  Therefore, we have
  \begin{align}
    \label{EdOm2d}
    \mathcal E_d(\chi_\delta \phi) - \int_\Omega \phi^2 d^3 r +
    {\delta^2 \over 2} \int_{\R^2} 
    \chi_\delta \bar\phi (-\Delta)^{1/2} \chi_\delta \bar \phi \, d^2
    r  \notag \\
    \geq - 5 \| \phi \|_{L^\infty(\Omega)}^2 |\Omega \backslash
    \Omega_{2 \delta}| - \delta^2 \int_\Omega |\nabla \phi|^2 d^3 r.
  \end{align}
  Noting that
  $| \Omega \backslash \Omega_{2 \delta}| \leq 6 |\partial D|
  \delta^2$
  for all $\delta > 0$ sufficiently small depending only on $D$ and
  combining \eqref{EdOm2d} with \eqref{Edbeta}, we finally arrive at
    \begin{align}
      \mathcal E_d(\phi) - \int_\Omega \phi^2 d^3 r +
      {\delta^2 \over 2} \int_{\R^2} 
      \chi_\delta \bar\phi (-\Delta)^{1/2} \chi_\delta \bar \phi \, d^2
      r  \geq - \beta \delta^2 \| \phi \|_{L^\infty(\Omega)}^2
      |\partial D| - \delta^2 \int_\Omega |\nabla \phi|^2 d^3 r,
    \end{align}
    for some universal $\beta > 0$. The lower bound in
    \eqref{EvsEElower} then follows by combining the above estimate
    with \eqref{Epoinc} and choosing $\alpha_2 = 1$.

    We now proceed to proving \eqref{EvsEEupper}. To begin, we define
    $\phi$ in $\Omega$ to be a $z$-independent function, thus,
    satisfying \eqref{av} in $D \times (0, \delta)$. Namely, for
    $(x, y, z) \in D \times (0, \delta)$, we define
    $\phi (x, y, z) := \bar \phi(x, y)$. Next, we extend $\phi$ to the
    rest of $\Omega$ by a reflection about
    $\partial D \times (0, \delta)$. More precisely, for
    $\mathbf r \in \R^2$ define
    $\rho(\mathbf r) : = \text{dist}(\mathbf r, \R^2 \backslash D) -
    \text{dist}(\mathbf r, D)$
    to be the signed distance function to $\partial D$ in the
    plane. Then, for all $\delta$ sufficiently small depending only on
    $D$ there is a tubular neighborhood of $\partial D$ in which we
    can define a continuous unit outward normal vector $\nu$ to the
    projection on $\partial D$, i.e., we have
    $\mathbf r + \rho(\mathbf r) \nu(\mathbf r) \in \partial D$ for
    all $\mathbf r \in D$ such that $|\rho(\mathbf r)| \leq \delta$
    and all $0 < \delta \leq \delta_0$ for some $\delta_0 > 0$
    depending only on $D$. We then define for $\mathbf r = (x, y) \in
    \R^2 \backslash D$ and all $z \in (0, \delta)$ the extension of
    $\bar\phi$ as $\phi(x, y, z) := \bar\phi(\mathbf r + 2 \nu(\mathbf
    r) \rho(\mathbf r))$. In view of the regularity of $\partial D$,
    we then have 
    \begin{align}
      \label{edgegrad}
      \int_{\Omega \backslash (D \times (0, \delta))} |\nabla \phi|^2
      d^3 r \leq 2 \delta \int_{D \backslash D_\delta} |\nabla \bar\phi|^2
      d^2 r,
    \end{align}
    for $\delta_0$ sufficiently small depending only on $D$. 

    We now use positivity of different terms in the energy and Lemma
    \ref{l:nloce} to estimate
    \begin{align}
      (1 - 2 \alpha \delta^2) \mathcal E(\phi) \leq \delta \int_D
      \left( \frac12 (1 - \alpha \delta^2) |\nabla \bar \phi|^2 +
      \frac14 \left( 1 - \bar\phi^2 \right)^2 \right) d^2 r \notag \\
      + \int_{\Omega \backslash (D \times (0, \delta))}
      \left( \frac12 |\nabla \phi|^2 + \frac14 \left( 1 - \phi^2
      \right)^2 \right) d^3 r \notag \\
      + {\gamma \over 2} \int_{\R^3} \left( \partial_z \phi
      (-\Delta)^{-1} \partial_z \phi - \phi^2 \right) \, d^3 r \notag \\
      + \gamma \alpha \delta^2 \int_\Omega \phi^2 d^3 r
      - {\alpha \delta^2 \over 2} \int_\Omega |\nabla \phi|^2 d^3 r . 
    \end{align}
    Accordingly, possibly increasing the values of $\alpha_1$,
    $\alpha_2$ and $\beta$, by Lemmas \ref{l:edge} and \ref{l:ded},
    Young's inequality and using \eqref{edgegrad}, we may write
    \begin{align}
      (1 - 2 \alpha \delta^2) \mathcal E(\phi) 
      & \leq E(\bar\phi) \delta
        + \delta \int_{D \backslash D_\delta} |\nabla \bar\phi|^2 d^2 r
        \notag \\
      & + \beta \delta^2 (1 + \gamma^2) \left( 1
        + \| \phi \|_{L^\infty(\Omega)}^4 \right) (
        |\partial D| + |D| \delta ).
    \end{align}
    The result then follows by possibly further decreasing the value
    of $\delta_0$ and increasing the value of $\beta$.
\end{proof}

\section{Proof of Theorem \ref{t:gioia}}
\label{sec:proof-theor-reft:g}

We begin by establishing compactness of sequences satisfying
\eqref{limsupE}. As in the statement of Theorem \ref{t:gioia}, for
$\phi_\delta \in H^1(\Omega^\delta)$ we define
$\bar \phi_\delta \in H^1(D)$ to be given by \eqref{av} with $\phi$
replaced by $\phi_\delta$. We also define $\Omega^\delta_\delta$,
etc., to be given by \eqref{Omd} with $\Omega$ replaced by
$\Omega^\delta$. 

\begin{proposition}
  \label{p:comp3d}
  For a sequence of $\delta \to 0$, assume
  $\phi_\delta \in H^1(\Omega^\delta)$ satisfies
  \eqref{limsupE}. Then, we have
  $\| \nabla \phi_\delta \|_{L^2({\Omega^\delta})}^2 \leq C \delta$
  for some $C > 0$ independent of $\delta$, and upon extraction of a
  subsequence $\bar\phi_\delta \rightharpoonup \bar\phi$ in $H^1(D)$
  and $\bar\phi_\delta \to \bar\phi$ in $L^p(D)$ for any
  $1 \leq p < \infty$.
\end{proposition}

\begin{proof}
  By Corollary \ref{c:nloc3dpos} and Cauchy-Schwarz inequality, we
  have
  \begin{align}
    C \delta 
    & \geq \mathcal E_\delta(\phi_\delta) \geq \int_{\Omega^\delta}
      \left( \frac12 |\nabla \phi_\delta|^2 - \frac12 (1 + \gamma)
      \phi_\delta^2 + \frac14 \phi_\delta^4 \right) d^3 r \notag \\
    & \geq \int_{\Omega^\delta}
      \left( \frac12 |\nabla \phi_\delta|^2 + \frac14 \phi_\delta^4
      \right) d^3 r - \frac12 (1 + \gamma) |{\Omega^\delta}|^{1/2} \left(
      \int_{\Omega^\delta} \phi_\delta^4 d^3 r \right)^{1/2} \notag \\
    & \geq \frac14 \int_{\Omega^\delta} |\nabla \phi_\delta|^2 d^3 r + \frac14
      \int_D \int_0^\delta |\nabla' \phi_\delta|^2 dz \, 
      d^2 r + \frac18 \int_{\Omega^\delta} \phi_\delta^4 \, d^3 r - \frac12 (1 +
      \gamma)^2 |{\Omega^\delta}|,   
  \end{align}
  where $\nabla' = (\partial_x, \partial_y, 0)$, for some $C > 0$
  independent of $\delta$. On the other hand, for all $\delta$
  sufficiently small we have $|{\Omega^\delta}| \leq 2 |D| \delta$. Hence, by
  Jensen's inequality and arguing by approximation we have
  \begin{align}
    C + (1 + \gamma)^2 |D|  \geq \frac{1}{4 \delta} \int_{\Omega^\delta} |\nabla
    \phi_\delta|^2 d^3 r + {1 \over 4 \delta^2} \int_D \left|
    \int_0^\delta \nabla' \phi_\delta \, dz \right|^2 d^2 r + \frac18
    \int_D \bar\phi_\delta^4 \, d^2 r \notag \\
    \geq \frac{1}{4 \delta} \int_{\Omega^\delta} |\nabla
    \phi_\delta|^2 d^3 r + \frac14 \int_D |\nabla \bar\phi_\delta|^2
    d^2 r + {1 \over 8 |D|} \left( \int_D \bar\phi_\delta^2 d^2 r\right)^2,
  \end{align}
  where we again used Cauchy-Schwarz inequality in the last step.
  Thus, the sequence of $\delta^{-1/2} |\nabla \phi_\delta|$ is
  bounded in $L^2({\Omega^\delta})$, the sequence of $\bar\phi_\delta$ is
  bounded in $H^1(D)$, and by compact embedding there exists a
  subsequence with the desired properties.
\end{proof}

We now turn to the proof of Theorem \ref{t:gioia}. We note that if we
also assume that $\| \phi_\delta \|_{L^\infty({\Omega^\delta})} \leq M$ for
some $M > 0$ independent of $\delta$, we could immediately combine the
result of Theorem \ref{t:main} with the result of the following
proposition to obtain the claim (however, see Remark
\ref{r:Linfty3d}).

\begin{proposition}
  \label{p:EdE0}
  Let $\bar \phi_\delta \in H^1(D)$, and assume that for a sequence of
  $\delta \to 0$ we have $\bar \phi_\delta \to \bar\phi$ in
  $L^2(D)$. Then 
  \begin{align}
    \label{liminfEEE0}
    \liminf_{\delta \to 0} E_\delta(\bar\phi_\delta) \geq E_0(\bar\phi).
  \end{align}
  Conversely, for any $\bar\phi \in H^1(D)$ we have
  \begin{align}
    \label{limsupEEE0}
    \limsup_{\delta \to 0} E_\delta(\bar\phi) \leq E_0(\bar\phi).
  \end{align}
\end{proposition}

\begin{proof}
  Without loss of generality, we may assume that 
  \begin{align}
    \label{EElimsup}
    \limsup_{\delta \to 0} E_\delta(\bar\phi_\delta) < +\infty.
  \end{align}
  Then, by Lemma \ref{l:2dsmdel} and Young's inequality the sequence
  of $\bar\phi_\delta$ is bounded in $H^1(D)$ and $L^4(D)$. Therefore,
  upon extraction of a subsequence, we also have
  $\bar\phi_\delta \rightharpoonup \bar\phi$ in $H^1(D)$. Arguing as
  in the proof of Lemma \ref{l:2dsmdel}, by lower semicontinuity of
  the $H^1(D)$ and $L^4(D)$ norms as well as strong convergence in
  $L^2(D)$ we then obtain \eqref{liminfEEE0}. Lastly, to obtain
  \eqref{limsupEEE0} we simply note that
  $E_\delta(\bar\phi) \leq E_0(\bar\phi)$.
\end{proof}

\begin{proof}[Proof of Theorem \ref{t:gioia}]
  The proof follows closely the arguments of the proof of Theorem
  \ref{t:main}, except we only apply the rough bound in
  \eqref{L2edge3drough}. The local part of the energy may be estimated
  exactly as in \eqref{Epoinc}. For the non-local part, we apply the
  first part of Lemma \ref{l:edge}. This leads to a lower bound
  \begin{align}
    \label{E2E2}
    \mathcal E_\delta(\phi_\delta) \geq 
    E_\delta(\bar\phi_\delta) \delta -2 \gamma
    \| \phi_\delta \|_{L^2({\Omega^\delta})} \| \phi_\delta 
    \|_{L^2({\Omega^\delta} \backslash \Omega^\delta_{2\delta})} + \frac14
    \int_{{\Omega^\delta} \backslash ( D \times (0, \delta))} \left( 1 -
    \phi_\delta^2 \right)^2 d^3 r.
  \end{align}

  To proceed, we note that by Poincar\'e's inequality we have
  \begin{align}
    \| \phi_\delta \|_{L^2(D \times (0, \delta))} \leq \delta^{1/2} \|
    \bar\phi_\delta \|_{L^2(D)} + {\delta \over \pi}  \| \nabla
    \phi_\delta \|_{L^2(D \times (0, \delta))},
  \end{align}
  and a similar estimate holds for
  $\| \phi_\delta \|_{L^2((D \backslash D_{2 \delta}) \times (0,
    \delta))}$.
  Therefore, from our assumption on the gradient of $\phi_\delta$ we
  obtain
  \begin{align}
    \label{phidL2Om}
    \| \phi_\delta \|_{L^2((D \times (0, \delta))} 
    & \leq \delta^{1/2} \| \bar\phi_\delta \|_{L^2(D)} + C
      \delta^{3/2}, \\ 
    \| \phi_\delta \|_{L^2((D \backslash D_{2 \delta}) \times (0, \delta))} 
    & \leq \delta^{1/2} \| \bar\phi_\delta \|_{L^2(D \backslash D_{2
      \delta})} + C \delta^{3/2},    \label{phidL2Omd}
  \end{align}
  for some $C > 0$ independent of $\delta$. Using these estimates, we
  get
  \begin{align}
    \| \phi_\delta \|_{L^2({\Omega^\delta})} \| \phi_\delta 
    \|_{L^2({\Omega^\delta} \backslash \Omega^\delta_{2\delta})} 
    & \leq \| \phi_\delta
      \|_{L^2({\Omega^\delta} \backslash (D \times (0, \delta)))}^2 \notag \\
    & + 2 \delta^{1/2}  \| \phi_\delta
      \|_{L^2({\Omega^\delta} \backslash (D \times (0, \delta)))} ( \|
      \bar\phi_\delta \|_{L^2(D)} + C \delta) \notag \\
    & + \delta  ( \|
      \bar\phi_\delta \|_{L^2(D)} + C \delta)  ( \|
      \bar\phi_\delta \|_{L^2(D \backslash D_{2 \delta})} + C \delta). 
  \end{align}
  Therefore, since $\bar\phi_\delta \to \bar\phi$ in $L^2(D)$, there
  is $C > 0$ such that
  \begin{align}
    \| \phi_\delta \|_{L^2({\Omega^\delta})} \| \phi_\delta 
    \|_{L^2({\Omega^\delta} \backslash \Omega^\delta_{2\delta})} 
    \leq \| \phi_\delta
    \|_{L^2({\Omega^\delta} \backslash (D \times (0, \delta)))}^2 
    + C \delta^{1/2}  \| \phi_\delta
    \|_{L^2({\Omega^\delta} \backslash (D \times (0, \delta)))} \notag \\
    + C \delta  ( \|
    \bar\phi_\delta \|_{L^2(D \backslash D_{2 \delta})} + \delta),
  \end{align}
  for all $\delta$ sufficiently small. Thus, by Cauchy-Schwarz
  inequality we have
  \begin{align}
    \label{L2L2}
    \| \phi_\delta \|_{L^2({\Omega^\delta})} \| \phi_\delta 
    \|_{L^2({\Omega^\delta} \backslash \Omega^\delta_{2\delta})} 
    & \leq C \delta ( \| \phi_\delta
      \|_{L^4({\Omega^\delta} \backslash (D \times (0, \delta)))}^2 \notag \\
    & + \| \phi_\delta
      \|_{L^4({\Omega^\delta} \backslash (D \times (0, \delta)))} + \|
      \bar\phi_\delta \|_{L^2(D \backslash D_{2 \delta})} + \delta),
  \end{align}
  for some $C > 0$ and all $\delta$ small enough.

  On the other hand, for $\delta$ sufficiently small depending only on
  $D$ we have
  \begin{align}
    \label{W4}
    \frac14 \int_{{\Omega^\delta} \backslash ( D \times (0, \delta))} \left( 1 -
    \phi_\delta^2 \right)^2 d^3 r \geq \int_{{\Omega^\delta} \backslash ( D
    \times (0, \delta))} \left( \frac18 \phi_\delta^4 - 1 \right) d^3
    r \notag \\
    \geq - 2 |\partial D| \delta^2 + \frac18 \| \phi_\delta
    \|_{L^4({\Omega^\delta} \backslash ( D
    \times (0, \delta)))}^4. 
  \end{align}
  Combining this estimate with \eqref{L2L2} and \eqref{E2E2}, we then
  get
  \begin{align}
    \mathcal E_\delta (\phi_\delta) - E_\delta(\bar \phi_\delta)
    \delta 
    & \geq \frac18 \| \phi_\delta
      \|_{L^4({\Omega^\delta} \backslash ( D
      \times (0, \delta)))}^4 - C \delta (\| \phi_\delta
      \|_{L^4({\Omega^\delta} \backslash (D \times (0, \delta)))}^2 \notag \\
    & + \| \phi_\delta
      \|_{L^4({\Omega^\delta} \backslash (D \times (0, \delta)))} + \|
      \bar\phi_\delta \|_{L^2(D \backslash D_{2 \delta})} + \delta)
      \notag \\
    & \geq - C' \delta^{4/3} - C \delta \|
      \bar\phi_\delta \|_{L^2(D \backslash D_{2 \delta})},
  \end{align}
  for some $C, C' > 0$ and all $\delta$ small enough. The lower bound
  in \eqref{liminfEg} then follows from Proposition \ref{p:EdE0} and
  the fact that
  $\| \bar\phi_\delta \|_{L^2(D \backslash D_{2 \delta})} \to 0$ as
  $\delta \to 0$. The latter is an immediate consequence of the strong
  convergence of $\bar\phi_\delta$ to $\bar \phi$ in $L^2(D)$.

  For the upper bound, we use the same construction as in the proof of
  Theorem \ref{t:main}. Let $\phi_\delta \in H^1(\Omega^\delta)$ be
  the function obtained from a given $\bar \phi \in H^1(D)$ in this
  way. Note that by construction we have 
  \begin{align}
    \| \phi_\delta \|_{L^2(\Omega^\delta \backslash (D \times (0,
    \delta)))}^2 
    & \leq 2 \delta \| \bar\phi \|_{L^2(D \backslash
      D_\delta)}^2, \\  
   \| \phi_\delta \|_{L^4(\Omega^\delta \backslash (D \times (0,
    \delta)))}^4 
    & \leq 2 \delta \| \bar\phi \|_{L^4(D \backslash
      D_\delta)}^4, \\  
    \| \nabla \phi_\delta
    \|_{L^2(\Omega^\delta \backslash (D \times (0, 
    \delta)))}^2 
    & \leq 2 \delta \| \nabla \bar\phi \|_{L^2(D \backslash
      D_\delta)}^2, 
  \end{align}
  for all $\delta$ sufficiently small depending only on $D$. In
  particular, we have
  $\| \nabla \phi_\delta \|_{L^2({\Omega^\delta})}^2 \leq C \delta$
  for some $C > 0$ independent of $\delta$. By Lemmas \ref{l:edge} and
  \ref{l:ded}, we get
  \begin{align}
    (1 - 2 \alpha \delta^2) \mathcal E_\delta(\phi_\delta) 
    & \leq E_\delta(\bar \phi) \delta + 2 \gamma
      \| \phi_\delta \|_{L^2({\Omega^\delta})} \| \phi_\delta 
      \|_{L^2({\Omega^\delta} \backslash \Omega^\delta_{2\delta})} + \gamma \alpha
      \delta^2 \| \phi \|_{L^2(\Omega^\delta_\delta)}^2 \notag \\ 
    & + \int_{{\Omega^\delta} \backslash (D \times (0, \delta))} \left( \frac12
      |\nabla \phi_\delta |^2 + \frac14 ( 1 - \phi_\delta^2)^2 \right)
      d^3 r.
  \end{align}
  Therefore, for $\delta$ sufficiently small depending only on $D$ we
  obtain
  \begin{align}
    \label{E2E2upp}
    (1 - 2 \alpha \delta^2) \mathcal E_\delta(\phi_\delta) 
    \leq E_\delta(\bar \phi) \delta + 6 \gamma \delta
    \| \bar\phi \|_{L^2(D)} \| \bar \phi 
    \|_{L^2(D \backslash D_{2\delta})} + \gamma \alpha
    \delta^3 \| \bar \phi \|_{L^2(D)}^2 \notag \\ 
    + \delta \int_{D \backslash D_\delta} \left( 
    |\nabla \bar \phi |^2 +  1 + \bar\phi^4 \right)
    d^2 r.
  \end{align}
  Note that the integral in the right-hand side of \eqref{E2E2upp}
  vanishes as $\delta \to 0$, since
  $\bar \phi \in H^1(D) \subset L^4(D)$ by Sobolev
  embedding. Similarly,
  $\| \bar\phi \|_{L^2(D \backslash D_{2 \delta})} \to 0$ as
  $\delta \to 0$. Thus, the estimate in \eqref{limsupEg} follows by
  Proposition \ref{p:EdE0}.
 \end{proof}

 \section{Rest of the proofs}
\label{sec:rest-proof}

We begin this section by presenting a brief demonstration of
Corollary \ref{c:nolteno}. Assume Theorem \ref{t:nolte} holds true. We
use $\bar\phi \equiv 1$ as an admissible test function for $E_*$ to
estimate the minimum energy from above. Then, if $\phi_\eps$ is a
minimizer of $\mathcal E_\eps$ and $\bar\phi_\eps$ is its $z$-average
given by \eqref{aveps}, by \eqref{Eepslb} and \eqref{Eepsub} we have
\begin{align}
  E_\eps(\bar \phi_\eps) \delta_\eps \leq \mathcal E_\eps(\phi_\eps) +
  O(\delta_\eps^2) \leq (1 - 2 \alpha
  \delta_\eps^2)^{-1} E_*(1) \delta_\eps + o(\delta_\eps),
\end{align}
as $\eps \to 0$. Thus, by the $\Gamma$-convergence of $E_\eps$ to
$E_*$ we get
\begin{align}
  \frac12 (\sigma_0 - \sigma_1 \lambda) \int_D |\nabla \bar\phi_\eps|
  \, d^2 r \to 0 \qquad \text{as} \quad \eps \to 0,
\end{align}
and in view of the fact that $\bar \phi_\eps = 1$ in
$D \backslash D_\rho$, we have $\bar \phi_\eps \to 1$ in $BV(D)$.

The proof of Theorem \ref{t:nolte} relies on a key interpolation lemma
that goes back to \cite{desimone06} and is generalized in
\cite{kmn17}, all in the periodic setting, to estimate the homogeneous
$H^{1/2}$ norm of $\bar \phi$ from above by the $L^\infty$ and the
$BV$ norms of $\bar \phi$. As was already pointed out in
\cite{desimone06}, this is impossible without an additional penalty
term due to the ``logarithmic failure'' of the corresponding embedding
\cite{desimone06}. Here we use the approach of \cite{kmn17} to extend
a version of the estimate in \cite[Lemma 4.1]{kmn17} to our setting,
noting that we need a nonlinear version of \cite[Lemma 4.1]{kmn17} in
order to combine it with the Modica-Mortola lower bound for the local
part of the energy.

\begin{lemma}
  \label{l:interp}
  Let $\bar\phi \in H^1(\R^2) \cap L^\infty(\R^2)$ be such that
  $\| \bar \phi\|_{L^\infty(\R^2)} \leq 1$ and
  $\mathrm{supp}(\bar\phi) \in B_R$. Then
  \begin{multline}
    \label{interp}
    {1 \over 4 \pi} \int_{\R^2} \int_{\R^2} {(\bar\phi(\mathbf r) -
      \bar\phi(\mathbf r'))^2 \over |\mathbf r - \mathbf r'|^3} \, d^2
    r \, d^2 r' \\
    \leq {3 \over \pi} \ln \left( {R \over r} \right) \| \nabla
    \left( \bar \phi - \tfrac13 \bar \phi^3 \right) \|_{L^1(\R^2)} +
    r \| \nabla \bar \phi \|_{L^2(\R^2)}^2 + \pi R,
  \end{multline}
  for any $r \in (0, R)$.
\end{lemma}

\begin{proof}
  The proof is a close adaptation of the proof of \cite[Lemma
  4.1]{kmn17}. Write the integral in \eqref{interp} as
  \begin{multline}
    \int_{\R^2} \int_{\R^2} {(\bar\phi(\mathbf r) - \bar\phi(\mathbf
      r'))^2 \over |\mathbf r - \mathbf r'|^3} \, d^2 r \, d^2 r' =
    \int_{B_{2R}} \int_{B_{2R}} {(\bar\phi(\mathbf r) -
      \bar\phi(\mathbf r'))^2 \over |\mathbf r - \mathbf r'|^3} \, d^2
    r \, d^2 r' \\
    + 2 \int_{B_{2R}} \int_{\R^2 \backslash B_{2R}}
    {\bar\phi^2(\mathbf r) \over |\mathbf r - \mathbf r'|^3} \, d^2 r'
    \, d^2 r \leq \int_{B_{2R}} \int_{B_{4R}} {(\bar\phi(\mathbf r +
      \mathbf z) - \bar\phi(\mathbf r))^2 \over |\mathbf z|^3} \, d^2
    z \, d^2 r \\
    + 2 \int_{B_R} \int_{\R^2 \backslash B_R} {\bar\phi^2(\mathbf r)
      \over |\mathbf z|^3} \, d^2 z \, d^2 r \\
    = \int_{B_{2R}} \int_{B_{4R}} {(\bar\phi(\mathbf r + \mathbf z) -
      \bar\phi(\mathbf r))^2 \over |\mathbf z|^3} \, d^2 z \, d^2 r +
    {4 \pi \over R} \| \bar\phi \|_{L^2(\R^2)}^2. \label{B2R4R}
  \end{multline}
  Focusing now on the first term above, we observe that by Jensen's
  inequality we have for all $\mathbf z \in B_{4R}$:
  \begin{align}
    \int_{B_{2R}} (\bar\phi(\mathbf r + \mathbf z) -
    \bar\phi(\mathbf r))^2 d^2 r \leq \int_{B_{2R}} \int_0^1 |\mathbf
    z \cdot \nabla
    \bar \phi(\mathbf r + t \mathbf z) |^2 \, dt \, d^2 r 
    \leq \int_{B_{6R}} |\mathbf z \cdot \nabla \bar \phi(\mathbf r)|^2 
    \, d^2 r. \label{phi2}
  \end{align}
  Similarly, introducing
  $\bar\psi := \bar \phi - \frac13 \bar \phi^3$, we have
  \begin{multline}
    \int_{B_{2R}} (\bar\phi(\mathbf r + \mathbf z) -
    \bar\phi(\mathbf r))^2 d^2 r \\
    \leq \int_{B_{2R} \cap \{ \bar \phi(\mathbf r + \mathbf z) \not=
      \bar \phi(\mathbf r) \} } \ {(\bar \phi(\mathbf r + \mathbf z) -
      \bar \phi(\mathbf r))^2 \over |\bar \phi(\mathbf r + \mathbf z)
      - \frac13 \bar \phi^3(\mathbf r + \mathbf z) - \bar \phi(\mathbf
      r ) + \frac13 \bar \phi^3(\mathbf r) |} \int_0^1 |\mathbf z
    \cdot \nabla \bar \psi(\mathbf r + t
    \mathbf z) | \, dt \, d^2 r  \\
    \leq 3 \int_{B_{6R}} |\mathbf z \cdot \nabla \bar \psi(\mathbf r)|
    \, d^2 r, \label{phi1}
  \end{multline}
  where we used the fact that 
  \begin{align}
    \left| { (s - t)^2 \over s - \frac13 s^3 - t + \frac13 t^3}
    \right| \leq 3
    \qquad \forall (s, t) \in (-1,1)^2, \ s \not= t,
  \end{align}
  which can be readily verified by means of elementary
  calculus. Indeed, for every $-1 < s < t < 1$ we have
  \begin{align}
    F(s, t) :=  { (s - t)^2 \over s - \frac13 s^3 - t + \frac13 t^3} =
    {3 (s - t) \over 3 - t^2 - ts - s^2},
  \end{align}
  and taking partial derivatives, we obtain
  \begin{align}
    {\partial F \over \partial t} = - {9 (1 - s^2) + 3 (s - t)^2 \over
    (3 - s^2 - s t - t^2)^2} < 0, \qquad  {\partial F \over \partial
    s} =  {9 (1 - t^2) + 3 (s - t)^2 \over (3 - s^2 - s t - t^2)^2} > 0. 
  \end{align}
  Hence $0 > F(s, t) > F(-1, 1) = -3$ for all $-1 < s < t < 1$. Since
  $F(s, t) = -F(t, s)$, we conclude that $|F(s, t)| \leq 3$ for all
  $(s, t) \in (-1,1)^2$ with $s \not= t$.

  Now, splitting the integral over $\mathbf z$ in \eqref{B2R4R} into a
  near-field part and a far-field part and using \eqref{phi2} and
  \eqref{phi1} to estimate the respective pieces, we get for any
  $0 < r < R$:
  \begin{multline}
    \int_{B_{2R}} \int_{B_{4R}} {(\bar\phi(\mathbf r + \mathbf z) -
      \bar\phi(\mathbf r))^2 \over |\mathbf z|^3} \, d^2 z \, d^2 r \\
    \leq \int_{B_{4 r}} \int_{B_{6R}} {|\mathbf z \cdot \nabla \bar
      \phi (\mathbf r)|^2 \over |\mathbf z|^3} \, d^2 r \, d^2 z + 3
    \int_{B_{4R} \backslash B_{4 r}} \int_{B_{6R}} {|\mathbf z
      \cdot \nabla \bar
      \psi (\mathbf r)| \over |\mathbf z|^3} \, d^2 r \, d^2 z  \\
    \leq 4 \pi r \| \nabla \bar \phi \|_{L^2(\R^2)}^2 + 12 \ln
    \left( { R \over r} \right) \| \nabla \bar \psi \|_{L^1(\R^2)}.
  \end{multline}
  Then, combining this estimate with \eqref{B2R4R}, we obtain the
  result.
\end{proof}

We point out that, importantly, the constant in front of the logarithm
in Lemma \ref{l:interp} is the best possible one (as was already
observed in \cite{desimone06,kmn17} in a slightly different setting),
which can be easily seen by considering the characteristic function of
$B_{R/2}$ mollified at scale $r$ as a test function, provided that $r$
is small enough.

We will also need a slightly modified version of Lemma
\ref{l:interp}. 

\begin{lemma}
  \label{l:interp2}
  Let $\bar \phi \in L^\infty(\R^2) \cap H^1_{loc}(\R^2)$ be such that
  $\bar \phi = 1$ in $\R^2 \backslash D$ and
  $\| \bar \phi \|_{L^\infty(\R^2)} = 1$. Then
  \begin{multline}
    \label{interp2}
    {1 \over 4 \pi} \int_{\R^2} \int_{\R^2} {(\bar\phi(\mathbf r) -
      \bar\phi(\mathbf r'))^2 \over |\mathbf r - \mathbf r'|^3} \, d^2
    r \, d^2 r' \\
    \leq {3 \over \pi} \ln \left( {R \over r} \right) \| \nabla
    \left( \bar \phi - \tfrac13 \bar \phi^3 \right) \|_{L^1(\R^2)} +
    r \| \nabla \bar \phi \|_{L^2(\R^2)}^2 + 4 \pi R,
  \end{multline}
  for some $R > 0$ and all $r \in (0, R)$.
\end{lemma}

\noindent The proof of Lemma \ref{l:interp2} is identical to that of
Lemma \ref{l:interp}. We note that the left-hand side in
\eqref{interp2} makes sense because $\bar \phi - 1 \in H^1(\R^2)$ and
can be interpreted as the homogeneous $H^{1/2}$ norm squared of
$\bar \phi - 1$.

\begin{proof}[Proof of Theorem \ref{t:nolte}]
  We begin with the proof of compactness. Let $\bar \phi_\eps$ be as
  in part (i) of Theorem \ref{t:nolte} and define
  $\bar \psi_\eps := \bar \phi_\eps - \frac13 \bar \phi_\eps^3$. Using
  the Modica-Mortola trick \cite{modica87} and weak chain rule
  \cite{evans}, we write for all $\eps$ sufficiently small
  \begin{align}
    E_\eps(\bar \phi_\eps) 
    & \geq {(\lambda_c + \lambda) \sqrt{1 -
      \alpha \delta_\eps^2} \over 2 \lambda_c
      \sqrt{2}} \int_D  |\nabla \bar
      \psi_\eps |  \, d^2 r \notag \\
    & + {\eps (\lambda_c - \lambda) (1 - \alpha \delta_\eps^2) \over 4
      \lambda_c} 
      \int_D |\nabla \bar \phi_\eps|^2 \, d^2 r + {\lambda_c -
      \lambda \over 8 \eps \lambda_c} 
      \int_D (1 - \bar \phi_\eps^2)^2 \, d^2 r \notag \\
    & - {\lambda \over 16 \pi |\ln \eps|} \int_{\R^2} \int_{\R^2} {
      (\chi_{\eps \delta_\eps} (\mathbf r) \bar \phi_\eps(\mathbf r) -
      \chi_{\eps \delta_\eps} (\mathbf r') \bar \phi_\eps(\mathbf r') )^2
      \over |\mathbf r - \mathbf r'|^3} \, d^2 r \, d^2 r'. 
  \end{align}  
  Applying Lemma \ref{l:interp} with $r = \eps \delta_\eps$, we,
  therefore, get
  \begin{multline}
    E_\eps(\bar \phi_\eps) + C \| \nabla \chi_{\eps \delta_\eps}
    \|_{L^1(D)} + C \eps \delta_\eps^2 \| \nabla \chi_{\eps
      \delta_\eps}
    \|_{L^2(D)}^2 + C \delta_\eps  \\
    \geq \frac14 (\sigma_0 - \lambda \sigma_1) \int_D |\nabla
    \bar \psi_\eps | \, d^2 r + {\lambda_c - \lambda \over 8 \eps
      \lambda_c} \int_D (1 - \bar \phi_\eps^2)^2 \, d^2 r, \label{mmrough}
  \end{multline}
  for some $C > 0$ and all $\eps$ small enough.  Since the left-hand
  side of the above expression is bounded as $\eps \to 0$, we obtain,
  upon extraction of a subsequence, that $|\bar \phi_\eps| \to 1$ in
  $L^1(D)$ and a.e. in $D$. Furthermore, by compactness in $BV$
  \cite{evans} , we have, upon extraction of another subsequence, that
  $\bar \psi_\eps \rightharpoonup \bar \psi$ in $BV(D)$, and
  $\bar \psi_\eps \to \bar \psi$ in $L^1(D)$ and a.e. in $D$, with
  $|\bar \psi| = \frac23$ a.e. in $D$. Thus, we get that
  $\bar \psi \in BV(D; \{ -\frac23, \frac23 \})$, which, in turn,
  implies that $\bar \phi_\eps \to \bar \phi $ in $L^1(D)$ with
  $\bar\phi = \frac32 \bar \psi \in BV(D; \{-1,1\})$. Also, clearly
  $\bar \phi = \bar \phi_\eps = 1$ in $D \backslash D_\rho$.

  We now prove the lower bound in \eqref{liminfEeps}. To that end, we
  make the estimate in \eqref{mmrough} quantitative by isolating the
  contribution of the edge to the non-local energy. We redefine
  $\bar \phi(x) := 1$ for all $x \in \R^2 \backslash D$ and introduce
  \begin{align}
    E_\eps^0(\bar \phi) := \int_D \left( \frac{\eps}{2} \left( 1 -
    \alpha \delta^2 \right) |\nabla \bar\phi|^2 + \frac{1}{4 \eps}
    \left( 1 - \bar\phi^2 \right)^2 \right) d^2 r 
    \qquad \qquad \notag \\ 
    - {\lambda \over 16 \pi |\ln \eps|} \int_{\R^2} \int_{\R^2}
    {(\bar\phi(\mathbf r) -
    \bar\phi(\mathbf r'))^2 \over |\mathbf r
    - \mathbf r'|^3} \, d^2 r \, d^2 r',
    \label{EE0eps}
  \end{align}
  which represents the energy $E_\eps$ without the contribution of the
  edges.  Then, since by our assumption $\bar\phi_\eps = 1$ in
  $\R^2 \backslash D_\rho$, we have
  $\chi_{\eps \delta_\eps} \bar \phi_\eps = \bar \phi_\eps - 1 +
  \chi_{\eps \delta_\eps}$
  for all $\eps$ sufficiently small and, therefore,
  \begin{multline}
    E_\eps(\bar \phi_\eps) = E_\eps^0(\bar \phi_\eps) - {\lambda \over
      16 \pi |\ln \eps|} \int_{\R^2} \int_{\R^2} {(\chi_{\eps
        \delta_\eps}(\mathbf r) - \chi_{\eps \delta_\eps}(\mathbf
      r'))^2 \over |\mathbf r -
      \mathbf r'|^3} \, d^2 r \, d^2 r' \\
    - {\lambda \over 8 \pi |\ln \eps|} \int_{\R^2} \int_{\R^2} {(\bar
      \phi_\eps(\mathbf r) - \bar \phi_\eps(\mathbf r')) (\chi_{\eps
        \delta_\eps}(\mathbf r) - \chi_{\eps \delta_\eps}(\mathbf r'))
      \over |\mathbf r - \mathbf r'|^3} \, d^2 r \, d^2 r'.
    \label{EE0epsbelow}
  \end{multline}
  Using Lemma \ref{l:interp2} and arguing as in \eqref{mmrough}, we
  can estimate
  \begin{align}
    E_\eps^0(\bar \phi_\eps) \geq \frac34 ( \sigma_0 - \lambda \sigma_1)
    \int_D |\nabla \bar \psi_\eps| \, d^2 r - {C \over
    |\ln \eps|},
  \end{align}
  for some $C > 0$ and all $\eps$ small enough. At the same time, by a
  direct computation as in the proof of \cite[Lemma 5.3]{kmn17} we
  have
  \begin{align}
    \int_{\R^2} \int_{\R^2} {(\chi_{\eps
    \delta_\eps}(\mathbf r) - \chi_{\eps \delta_\eps}(\mathbf
    r'))^2 \over |\mathbf r - \mathbf r'|^3} \, d^2 r \, d^2 r' \leq
    4 \, |\partial D| \, |\ln \eps| + C \ln |\ln \eps|,
  \end{align}
  for some $C > 0$ and $\eps$ small enough. Finally, we estimate the
  integral involving the mixed term in \eqref{EE0epsbelow} for all
  $\eps$ so small that $\chi_{\eps \delta_\eps} = 1$ in $D_{\rho/2}$: 
  \begin{multline}
    \int_{\R^2} \int_{\R^2} {(\bar \phi_\eps(\mathbf r) - \bar
      \phi_\eps(\mathbf r')) (\chi_{\eps \delta_\eps}(\mathbf r) -
      \chi_{\eps \delta_\eps}(\mathbf r')) \over |\mathbf r - \mathbf
      r'|^3} \, d^2 r \, d^2 r' \\
    = 2 \int_{D_\rho} \int_{\R^2 \backslash D_{\rho/2}} {(\bar
      \phi_\eps (\mathbf r) - 1) (1 - \chi_{\eps \delta_\eps}(\mathbf
      r')) \over
      |\mathbf r - \mathbf r'|^3} \, d^2 r' \, d^2 r  \\
    \leq 8 \int_{D_\rho} \left( \int_{\R^2 \backslash
        B_{\rho/2}(\mathbf r)} {1 \over |\mathbf r - \mathbf r'|^3} \,
      d^2 r' \right) d^2 r \leq {32 \pi |D| \over \rho}.
  \end{multline}
  Putting all these estimates together, we then obtain
  \begin{align}
    E_\eps(\bar \phi_\eps) \geq \frac34 ( \sigma_0 - \lambda
    \sigma_1 )
    \int_D |\nabla \bar \psi_\eps| \, d^2 r - {\lambda \sigma_1 \over
    4}  |\partial D| - C \, {\ln |\ln \eps| \over |\ln \eps|},
  \end{align}
  for some $C > 0$ and all $\eps$ small enough. The proof is concluded
  from the lower semicontinuity of the total variation \cite{evans}
  and the fact that $\bar \phi = \frac32 \bar \psi$. 

  Finally, the upper bound in \eqref{limsupEeps} follows from the
  standard construction of the recovery sequence for the
  Ginzburg-Landau energy exactly as in \cite[Lemma 5.3]{kmn17}
\end{proof}

\paragraph{Acknowledgements}
This work was supported, in part, by NSF via grants DMS-1313687 and
DMS-1614948. The author would like to thank V. Slastikov for many
valuable comments.

\bibliographystyle{plain}

\bibliography{../nonlin,../stat,../mura}

\begin{thebibliography}{10}

\bibitem{agmon59}
S.~Agmon, A.~Douglis, and L.~Nirenberg.
\newblock Estimates near the boundary for solutions of elliptic partial
  differential equations satisfying general boundary conditions.
\newblock {\em Commun. Pure Appl. Math.}, 12:623--727, 1959.

\bibitem{andelman87}
D.~Andelman, F.~Bro{\c c}hard, and J.‐F. Joanny.
\newblock Phase transitions in {Langmuir} monolayers of polar molecules.
\newblock {\em J. Chem. Phys.}, 86:3673--3681, 1987.

\bibitem{andelman09}
D.~Andelman and R.~E. Rosensweig.
\newblock Modulated phases: Review and recent results.
\newblock {\em J. Phys. Chem. B}, 113:3785--3798, 2009.

\bibitem{braides}
A.~Braides.
\newblock {\em {$\Gamma$}-convergence for beginners}, volume~22 of {\em Oxford
  Lecture Series in Mathematics and its Applications}.
\newblock Oxford University Press, Oxford, 2002.

\bibitem{braides08}
A.~Braides and L.~Truskinovsky.
\newblock Asymptotic expansions by {$\Gamma$}-convergence.
\newblock {\em Continuum Mech. Thermodyn.}, 20:21--62, 2008.

\bibitem{chikazumi}
S.~Chikazumi.
\newblock {\em Physics of Ferromagnetism}.
\newblock Oxford Univ. Press, 2005.

\bibitem{choksi98}
R.~Choksi and R.~V. Kohn.
\newblock Bounds on the micromagnetic energy of a uniaxial ferromagnet.
\newblock {\em Comm. Pure Appl. Math.}, 51:259--289, 1998.

\bibitem{choksi99}
R.~Choksi, R.~V. Kohn, and F.~Otto.
\newblock Domain branching in uniaxial ferromagnets: a scaling law for the
  minimum energy.
\newblock {\em Commun. Math. Phys.}, 201:61--79, 1999.

\bibitem{demasi94}
A.~De~Masi, E.~Orlandi, E.~Presutti, and L.~Triolo.
\newblock Glauber evolution with the {K}ac potentials. {I}. {M}esoscopic and
  macroscopic limits, interface dynamics.
\newblock {\em Nonlinearity}, 7:633--696, 1994.

\bibitem{desimone06}
A.~DeSimone, H.~Kn\"upfer, and F.~Otto.
\newblock 2-d stability of the {N\'eel} wall.
\newblock {\em Calc. Var. PDE}, 27:233--253, 2006.

\bibitem{desimone06r}
A.~DeSimone, R.~V. Kohn, S.~M\"uller, and F.~Otto.
\newblock Recent analytical developments in micromagnetics.
\newblock In G.~Bertotti and I.~D. Mayergoyz, editors, {\em The Science of
  Hysteresis}, volume~2 of {\em Physical Modelling, Micromagnetics, and
  Magnetization Dynamics}, pages 269--381. Academic Press, Oxford, 2006.

\bibitem{dinezza12}
E.~Di~Nezza, G.~Palatucci, and E.~Valdinoci.
\newblock Hitchhiker's guide to the fractional {S}obolev spaces.
\newblock {\em Bull. Sci. Math.}, 136:521--573, 2012.

\bibitem{evans}
L.~C. Evans and R.~L. Gariepy.
\newblock {\em Measure Theory and Fine Properties of Functions}.
\newblock CRC, Boca Raton, revised edition, 2015.

\bibitem{garel82}
T.~Garel and S.~Doniach.
\newblock Phase transitions with spontaneous modulations --- the dipolar
  {Ising} ferromagnet.
\newblock {\em Phys. Rev. B}, 26:325--329, 1982.

\bibitem{gilbarg}
D.~Gilbarg and N.~S. Trudinger.
\newblock {\em Elliptic Partial Differential Equations of Second Order}.
\newblock Springer-Verlag, Berlin, 1983.

\bibitem{gioia97}
G.~Gioia and R.~D. James.
\newblock Micromagnetics of very thin films.
\newblock {\em Proc. R. Soc. Lond. Ser. A}, 453:213--223, 1997.

\bibitem{hohenberg15}
P.C. Hohenberg and A.P. Krekhov.
\newblock An introduction to the {Ginzburg--Landau} theory of phase transitions
  and nonequilibrium patterns.
\newblock {\em Phys. Rep.}, 572:1--42, 2015.

\bibitem{hubert}
A.~Hubert and R.~Sch\"afer.
\newblock {\em Magnetic Domains}.
\newblock Springer, Berlin, 1998.

\bibitem{jagla04}
E.~A. Jagla.
\newblock Numerical simulations of two-dimensional magnetic domain patterns.
\newblock {\em Phys. Rev. E}, 70:046204, 2004.

\bibitem{kaplan93}
B.~Kaplan and G.~A. Gehring.
\newblock The domain structure in ultrathin magnetic films.
\newblock {\em J. Magn. Magn. Mater.}, 128:111--116, 2993.

\bibitem{km:jns11}
H.~Kn\"upfer and C.~B. Muratov.
\newblock Domain structure of bulk ferromagnetic crystals in applied fields
  near saturation.
\newblock {\em J. Nonlinear Sci.}, 21:921--962, 2011.

\bibitem{kmn17}
H.~Kn\"upfer, C.~B. Muratov, and F.~Nolte.
\newblock Magnetic domains in thin ferromagnetic films with strong
  perpendicular anisotropy.
\newblock Preprint, 2017.

\bibitem{kohn05arma}
R.~V. Kohn and V.~V. Slastikov.
\newblock Another thin-film limit of micromagnetics.
\newblock {\em Arch. Ration. Mech. Anal.}, 178:227--245, 2005.

\bibitem{landau8}
L.~D. Landau and E.~M. Lifshits.
\newblock {\em Course of Theoretical Physics}, volume~8.
\newblock Pergamon Press, London, 1984.

\bibitem{lieb-loss}
E.~H. Lieb and M.~Loss.
\newblock {\em Analysis}.
\newblock American Mathematical Society, Providence, RI, 2010.

\bibitem{lmm:jns15}
J.~Lu, V.~Moroz, and C.~B. Muratov.
\newblock Orbital free density functional theory of out-of-plane charge
  screening in graphene.
\newblock {\em J. Nonlinear Sci.}, 25:1391--1430, 2015.

\bibitem{lunardi}
A.~Lunardi.
\newblock {\em Analytic semigroups and optimal regularity in parabolic
  problems}, volume~16 of {\em Progress in Nonlinear Differential Equations and
  their Applications}.
\newblock Birkh\"auser, Basel, 1995.

\bibitem{malozemoff}
A.~P. Malozemoff and J.~C. Slonczewski.
\newblock {\em Magnetic Domain Walls in Bubble Materials}.
\newblock Academic Press, New York, 1979.

\bibitem{modica87}
L.~Modica.
\newblock The gradient theory of phase transitions and the minimal interface
  criterion.
\newblock {\em Arch. Rational Mech. Anal.}, 98:123--142, 1987.

\bibitem{mourrat16}
J.-C. Mourrat and H.~Weber.
\newblock Convergence of the two-dimensional dynamic ising-kac model to
  {$\Phi^4_2$}.
\newblock {\em Comm. Pure Appl. Math.}, 2016 (published online).

\bibitem{m:phd}
C.~B. Muratov.
\newblock {\em Theory of domain patterns in systems with long-range
  interactions of Coulombic type}.
\newblock Ph. D. Thesis, Boston University, 1998.

\bibitem{m:pre02}
C.~B. Muratov.
\newblock Theory of domain patterns in systems with long-range interactions of
  {Coulomb} type.
\newblock {\em Phys. Rev. E}, 66:066108 pp. 1--25, 2002.

\bibitem{m:cmp10}
C.~B. Muratov.
\newblock Droplet phases in non-local {Ginzburg-Landau} models with {Coulomb}
  repulsion in two dimensions.
\newblock {\em Comm. Math. Phys.}, 299:45--87, 2010.

\bibitem{mov:jap15}
C.~B. Muratov, V.~V. Osipov, and E.~Vanden-Eijnden.
\newblock Persistence of magnetization configurations against thermal noise in
  thin ferromagnetic nanorings with four-fold magnetocrystalline anisotropy.
\newblock {\em J. Appl. Phys.}, 117:17D118, 2015.

\bibitem{ng95}
K.-O. Ng and D.~Vanderbilt.
\newblock Stability of periodic domain structures in a two-dimensional dipolar
  model.
\newblock {\em Phys. Rev. B}, 52:2177--2183, 1995.

\bibitem{otto10}
F.~Otto and T.~Viehmann.
\newblock Domain branching in uniaxial ferromagnets: asymptotic behavior of the
  energy.
\newblock {\em Calc. Var. Partial Differential Equations}, 38:135--181, 2010.

\bibitem{roland90}
C.~Roland and R.~C. Desai.
\newblock Kinetics of quenched systems with long-range repulsive interactions.
\newblock {\em Phys. Rev. B}, 42:6658--6669, 1990.

\bibitem{rosensweig}
R.~E. Rosensweig.
\newblock {\em Ferrohydrodynamics}.
\newblock Courier Dover Publications, 1997.

\bibitem{seul95}
M.~Seul and D.~Andelman.
\newblock Domain shapes and patterns: the phenomenology of modulated phases.
\newblock {\em Science}, 267:476--483, 1995.

\bibitem{strukov}
B.~A. Strukov and A.~P. Levanyuk.
\newblock {\em Ferroelectric Phenomena in Crystals: Physical Foundations}.
\newblock Springer, New York, 1998.

\end{thebibliography}

\end{document}